\documentclass[final]{siamltex}

\usepackage[dvips]{color}
\usepackage{amsfonts}
\usepackage{dsfont}
\usepackage{amsmath,amssymb}
\usepackage[T1]{fontenc}
\usepackage{ae}
\usepackage{epic,eepic}
\usepackage{longtable}
\usepackage{rotating}
\usepackage{epsfig,graphicx}
\usepackage{exscale}
\usepackage{relsize}
\usepackage{booktabs}

          \newcommand{\HRule}{\rule{\linewidth}{0.5mm}}

\newtheorem{thm}{Theorem}[section]
\newtheorem{cor}[thm]{Corollary}
\newtheorem{lem}[thm]{Lemma}
\newtheorem{prop}[thm]{Proposition}

\newtheorem{rem}[thm]{Remark}
\newtheorem{ex}[thm]{Example}

\renewcommand{\refname}{\textbf{References}}
\newcommand{\C}{{\mathbb C}}

\newcommand{\R}{{\mathbb R}}

\renewcommand{\Im}{\operatorname{Im}}

\newcommand{\tr}{\operatorname{tr}}

\title{Riemannian optimization\\ on tensor products of Grassmann manifolds: Applications to generalized Rayleigh-quotients}

\author{O. Curtef\footnotemark[2]$\;\;$\footnotemark[3]
\and G. Dirr\footnotemark[2]\ 
\and U. Helmke\footnotemark[2]}

\begin{document}
\maketitle

\renewcommand{\thefootnote}{\fnsymbol{footnote}}
\footnotetext[2]{Department of Mathematics, University of W\"{u}rzburg, 97074
  W\"{u}rzburg, Germany}
  \footnotetext[3]{Corresponding author: curtef@mathematik.uni-wuerzburg.de}
\begin{abstract}
We introduce a generalized Rayleigh-quotient $\rho_A$ on the
tensor product of Grassmannians $\mathrm{Gr}^{\otimes r}({\bf m},{\bf n})$
enabling a unified approach to well-known optimization tasks from
different areas of numerical linear algebra,
such as best low-rank approximations of tensors (data compression),
geometric measures of entanglement (quantum computing) and subspace
clustering (image processing).
We briefly discuss the
geometry of the constraint set
$\mathrm{Gr}^{\otimes r}({\bf m},{\bf n})$,
we compute the Riemannian gradient of $\rho_A$,
we characterize its critical points and prove that they are generically non-degenerated. Moreover, we derive an explicit  necessary
condition for the non-degeneracy of the
Hessian. Finally, we present two intrinsic
methods for optimizing $\rho_A$ --- a Newton-like and a conjugated
gradient --- and
compare our algorithms tailored to the above-mentioned
applications with established ones from the literature.
\end{abstract}

\begin{keywords} Riemannian optimization, Grassmann manifold,
multilinear rank, best approximation of tensors, subspace clustering,
entanglement measure, Newton method, conjugated gradient method.
\end{keywords}

\begin{AMS}
14M15, 15A69, 65D19, 65F99, 65K10, 81P68
\end{AMS}

\pagestyle{myheadings} \thispagestyle{plain}
\markboth{O. CURTEF, G. DIRR AND U. HELMKE}{OPTIMIZATION ON TENSOR
  PRODUCTS OF GRASSMANNIANS}

\section{\bf Introduction}
The present paper addresses a constrained optimization problem,
subsuming and extending optimization tasks which arise in various
areas of applications such as
(i) low-rank tensor approximation problems from signal
processing and data compression, (ii) geometric measures of
pure state entanglement from quantum computing, (iii) subspace
reconstruction problems from image processing and (iv) combinatorial problems.

The problem can be stated as follows:
Given a collection of integer pairs $(m_j,n_j)$ with $1\leq m_j \leq n_j$
for $j = 1, \dots, r$ and a Hermitian $N\times N$ matrix $A$ with
$N := n_1n_2\cdots n_r$, find the global maximizer of the trace function
${\bf P} \mapsto \tr(A{\bf P})$. Here, ${\bf P}$ is restricted to
the set of all Hermitian projectors ${\bf P}:\C^{N} \to \C^{N}$ of rank
$M:= m_1m_2\cdots m_r$, which can be represented as a tensor product
${\bf P}:=P_1\otimes\dots\otimes P_r$ of
Hermitian projectors $P_j: \C^{n_j} \to \C^{n_j}$ of rank $m_j$.
Thus, one is faced with the constrained optimization task
\begin{equation}\label{gRq1}
\max  \tr(A{\bf P})
\quad
\textrm{subject to}\; {\bf P}\in \mathrm{Gr}^{\otimes r}({\bf m},{\bf n}),
\end{equation}
where $\mathrm{Gr}^{\otimes r}({\bf m},{\bf n})$
denotes the set of all Hermitian projectors of the above tensor type
and $({\bf m},{\bf n})$ is a shortcut for
$\big((m_1,n_1),\dots, (m_r,n_r)\big)$.
We will see that it makes sense to call the above
objective function ${\bf P} \mapsto \tr(A{\bf P})=:\rho_A({\bf P})$
the \emph{generalized Rayleigh-quotient} of $A$ with respect
to the partitioning $({\bf m},{\bf n})$.

To the best of the authors' knowledge, problem (\ref{gRq1}) has not
been discussed in the literature in this general setting.
However, depending on the structure of $A$ as well as on the choice of
$({\bf m},{\bf n})$, problem (\ref{gRq1}) relates
to well-known numerical linear algebra issues:

\smallskip

(i) For Hermitian matrices of rank-$1$, i.e. $A=vv^{\dagger}$,
it reduces to a best low-rank approximation problem for the tensor
$\mathcal{A}\in\C^{n_1\times n_2\times\dots\times n_r}$
which satisfies
$v=\mathrm{vec}(\mathcal{A})$, cf.~\cite{lmv2,sl}. Classical
application areas of such low-rank approximations can be found
in statistics, signal processing and data compression
\cite{pc,lmv1,lmv2,tlr}.

(ii) A recent application in quantum computing
plays a central role in characterizing and quantifying pure state
entanglement. Here, the distance
of a pure state (tensor)
to the set of all product states
(rank-$1$ tensors) provides a geometric measure for entanglement
\cite{dhk,nc,wg}.

(iii) Moreover, the challenging task of recovering subspaces of
possibly different dimensions from noisy data --- known as
subspace detection or subspace clustering problem
in computer vision and image processing \cite{vmp} --- can also be
cast into the above setting. More precisely,
for an appropriately chosen Hermitian matrix $A$
the subspace clustering task can be 
characterized by problem (\ref{gRq1}) in the sense that
for unperturbed data the global minima of the generalized
Rayleigh-quotient  are in unique correspondence with the sought
subspaces. Numerical experiments in Section 4 support
that even for noisy data the proposed optimization
yields reliable approximations of the unperturbed subspaces.

(iv) In \cite{broket}  a certain class of combinatorial problems are  recast as  optimization problems for trace functions on the special unitary group. For the case when $A$ is a diagonal matrix, optimization task \eqref{gRq1} is a generalization of the applications mentioned in \cite{broket}.

\smallskip

Our solution to problem ($\ref{gRq1}$) is based on the fact that the
constraint set $\mathrm{Gr}^{\otimes}({\bf m},{\bf n})$ can be
equipped with a Riemannian submanifold structure. This admits
the use of techniques from Riemannian  optimization --- a rather
new approach towards constrained optimization exploiting the
geometrical structure of the constraint set in order to develop
numerical algorithms \cite{ams, hm, udriste}.
In particular, we pursue two approaches: a Newton and a conjugated
gradient method.

On a Riemannian manifold, the intrinsic Newton method is
usually described by means of the Levi-Civita connection, performing
iterations along geodesics, see \cite{gd, st}.
A more general approach 
via local coordinates was initiated by Shub in \cite{shub} and further
discussed in \cite{ams, hht}.  Here, we follow the ideas in
\cite{hht} and use a pair of local parametrizations --- normal
coordinates for the push-forward and QR-type coordinates for the
pull-back --- satisfying an additional compatibility condition to
preserve quadratic convergence. 
Thus we obtain an intrinsically defined version of the classical Newton
algorithm with some computational flexibility.
Nevertheless, for high-dimensional problems its iterations are expensive, both in terms of computational complexity
and memory requirements. 
Therefore, we alternatively propose a conjugated gradient method,
which has the advantage of algorithmic simplicity at
a satisfactory
convergence rate. In doing so, we suggest to replace the global
line-search of the classical conjugated gradient method
by a one-dimensional Newton-step, which yields a better convergence
behavior near stationary points than the commonly used Armijo-rule.

As mentioned earlier, depending on the structure of $A$, the
above-specified problems  (i), (ii), (iii) and (iv) are particular cases
of  the optimization task ($\ref{gRq1}$).
For the best low-rank approximation of a tensor the standard
numerical approach 
is an alternating least-squares algorithm, known as higher-order
orthogonal iteration (HOOI) \cite{lmv2}. Recently, several new
methods also exploiting the geometric structure of the problem
have been published.
Newton algorithms have been proposed in \cite{es,ilav2},
quasi-Newton methods in \cite{sl}, conjugated gradient and trust
region methods in \cite{ilav1}. For high-dimensional tensors, all
Riemannian Newton algorithms manifest similar problems: too high
computational complexity and memory requirements.
Our conjugated gradient method is however, a good candidate to solve large scale problems. It exhibits locally a good convergence behavior, comparable to that of the quasi-Newton methods in \cite{sl} at much lower
computational costs, which considerably reduces the necessary 
CPU time.

For the problem of estimating a mixture of linear subspaces from
sampled data points, cf.~(iii), our numerical approach is an efficient
alternative to the classical ones: \emph{ad-hoc} type methods such as
K-subspace algorithms \cite{ho}, or probabilistic methods using a Maximum
Likelihood framework for the  estimation \cite{tb}.

\smallskip

The paper is organized as follows.
In  Section 2, we familiarize the reader with the basic ingredients
of Riemannian optimization.
In particular, we address the following topics: the Riemannian
submanifold structure of the constraint set
$\mathrm{Gr}^{\otimes r}({\bf m},{\bf n})$,
its isometry to the $r$-fold cartesian product of
Grassmannians, geodesics and parallel transport
and the computation of the intrinsic gradient and
Hessian for smooth objective functions.
Section 3 is dedicated to the problem of optimizing the generalized
Rayleigh-quotient $\rho_A$, including also a
detailed discussion on its relation to problems (i), (ii), (iii) and (iv).
Moreover, an analogy to the classical Rayleigh-quotient is also the subject of this section.
We compute the gradient and the Hessian
of the generalized Rayleigh-quotient and derive
critical point conditions. We end the section with a result on the generic non-degeneracy of its critical points.
In Section 4, a Newton-like and a conjugated gradient algorithm
as well as numerical simulations tailored to the previously mentioned
applications are given.

\section{\bf Preliminaries}

\subsection{\bf Riemannian structure of
$\mathrm{Gr}^{\otimes r}({\bf m},{\bf n})$}\label{Riem_geom}

We start our  
study on the optimization task ($\ref{gRq1}$)
with a brief  summary on the necessary notations
and basic concepts.

Let  $\mathfrak{her}_n$ be  the set of all  Hermitian $n\times n$ matrices $A$,
i.e. $A\in\C^{n\times n}\;\textrm{with}\; A^{\dagger}=A$, where $A^{\dagger}$
refers to the conjugate  transpose of $A.$ Moreover, let $\mathrm{SU}_n$ be the Lie group of all special unitary
matrices  and  $\mathfrak{su}_n$ its Lie-algebra,
i.e. $\Theta\in\mathrm{SU}_n$ if and only if
$\Theta^{\dagger}\Theta=I_n,\;\mathrm{det}\Theta =1$
and, respectively,
$\Omega\in\mathfrak{su}_n$  if and only if
$\Omega^{\dagger}=-\Omega$ and $\tr(\Omega)=0.$
The \emph{Grassmannian},
\begin{equation}\label{gr} \mathrm{Gr}_{m,n}:=\{P\in\C^{n\times
    n}\;|\;P=P^{\dagger}=P^2,\;\tr(P)=m\},\end{equation}
is the set of all rank $m$ Hermitian projection operators of $\C^n$. It is a smooth and compact submanifold of $\mathfrak{her}_n$ with  real dimension
$2m(n-m)$, whose tangent space at $P$ is given by
\begin{equation}\label{tangent}\mathrm{T}_{P}\mathrm{Gr}_{m,n}=\{[P,\Omega]:=P\Omega-\Omega P\;|\;\Omega\in\mathfrak{su}_{n}\},\end{equation}
cf.~\cite{hht}. Hence, every element $P\in\mathrm{Gr}_{m,n}$ and every tangent vector $\xi\in\mathrm{T}_{P}\mathrm{Gr}_{m,n}$ can be
written as
\begin{equation}\label{Pxi}
P=\Theta\Pi_{m,n}\Theta^{\dagger}\;\;\textrm{and}\;\;\xi=\Theta\zeta_{m,n}\Theta^{\dagger},
\end{equation} where $\Pi_{m,n}$ is the \emph{standard
projector} of rank $m$ acting on $\C^n$ and $\zeta_{m,n}$ denotes  a tangent
vector in the corresponding tangent space, i.e.
\begin{equation}\label{P_s}
\begin{array}{cc}
\Pi_{m,n}=\left[\begin{array}{cc}
I_m & 0\\
0 & 0
\end{array}\right], & \zeta_{m,n}=\left[\begin{array}{cc}
0 & Z\\
Z^{\dagger} & 0
\end{array}\right],\;\; Z\in\C^{m\times (n-m)}.
\end{array}
\end{equation} Whenever the values of $m$ and $n$ are clear from the context,
we will use the shortcuts $\Pi$ and $\zeta$.
With respect to the Riemannian metric induced by the Frobenius inner
product of $\mathfrak{her}_n$, the Grassmannian $\mathrm{Gr}_{m,n}$ is a
Riemannian submanifold and the unique  orthogonal projector onto
$\mathrm{T}_P\mathrm{Gr}_{m,n}$ is given by
\begin{equation}\label{ad}
\mathrm{ad}^2_P X=[P,[P,X]],\;\;\; X\in\mathfrak{her}_n.
\end{equation}

We define  the  \emph{ $r-$fold tensor product of
  Grassmannians} $\mathrm{Gr}_{m_j,n_j},\; j=1,\dots, r$
as the set \begin{equation}
 \mathrm{Gr}^{\otimes r}({\bf m}, {\bf n}):=\displaystyle \{ P_1\otimes\dots\otimes P_r\;|\; P_j\in\mathrm{ Gr}_{m_j,n_j},\; j=1,\dots,r\}
\end{equation}
of all rank-$M$ Hermitian projectors 
of $\mathbb{C}^N$ with $M:= m_1m_2\cdots m_r$ and $N:=n_1n_2\cdots n_r$,
which can be represented as a 
Kronecker product $P_1\otimes\dots\otimes P_r.$
Here, $({\bf m},{\bf n})$ stands for the multi index
\begin{equation}\label{mn}
({\bf m},{\bf n}):=\Big((m_1,n_1), (m_2,n_2), \dots, (m_r,n_r)\Big).
\end{equation} Then, $\mathrm{Gr}^{\otimes r}({\bf m}, {\bf n})$ can be
naturally equipped with a submanifold structure as the following result
shows.
\begin{prop}
The  $r-$fold tensor product of Grassmannians $\mathrm{Gr}^{\otimes r}({\bf m}, {\bf n})$ is a smooth and compact
submanifold of $\mathfrak{her}_N$
of real dimension $2\displaystyle\sum\limits_{i=1}^rm_i(n_i-m_i)$.
\end{prop}
\begin{proof}
We consider the following smooth action
\begin{equation*}
\begin{array}{cc}\sigma: \mathrm{SU}({\bf n})\times
 \mathfrak{her}_N\rightarrow \mathfrak{her}_N, &
({\boldsymbol \Theta},Y)\mapsto {\boldsymbol \Theta}Y{\boldsymbol
  \Theta}^{\dagger},
\end{array}\end{equation*}
of the compact Lie group
\begin{equation}\label{SUloc}\mathrm{SU}({\bf n}):=\{{\boldsymbol
  \Theta}:=\displaystyle\Theta_1\otimes\dots\otimes \Theta_r \;| \; \Theta_j\in\mathrm{SU}_{n_j}\}\subset \mathrm{SU}_N.\end{equation}
Let $X\in \mathfrak{her}_N$ be of the form $X:=\displaystyle \Pi_1\otimes
\cdots\otimes \Pi_r,$ where $\Pi_j$ denotes the standard projector in $\mathrm{Gr}_{m_j,n_j}.$
Then, the orbit $\mathcal{O}(X):=\{{\boldsymbol \Theta} X {\boldsymbol
  \Theta}^{\dagger} | \; {\boldsymbol \Theta}\in\mathrm{SU}({\bf n})\}$ of $X$ coincides with
$\mathrm{Gr}^{\otimes r}({\bf m}, {\bf n})$. By \cite{hm} (pp. 44--46) we conclude that
the $r-$fold tensor product of Grassmannians is a smooth and compact
submanifold of $\mathfrak{her}_N$.
Moreover, $\mathcal{O}(X)\cong
\mathrm{SU}({\bf n})/\mathrm{Stab}(X)$,  where the stabilizer subgroup of $X$ is given by \begin{equation*}
\begin{array}{ll}
\mathrm{Stab}(X)  & :=\{{\boldsymbol \Theta}\in\mathrm{SU}({\bf n})\;|\; {\boldsymbol\Theta}X{\boldsymbol\Theta}^{\dagger}=X\}\\[1mm]
   & = \{{\boldsymbol \Theta}\in\mathrm{SU}({\bf n})\;|\; \Theta_j\Pi_j\Theta_j^{\dagger}=\Pi_j,\; j=1,\dots, r\}.
 \end{array}
\end{equation*}
It follows easily that the dimension of $\mathrm{Stab}(X)$ is $\displaystyle\sum\limits_{i=1}^r \,[m_i^2+(n_i-m_i)^2-1]$ and therefore,
$$\mathrm{dim}~(\mathrm{Gr}^{\otimes r}({\bf m}, {\bf n})) = \mathrm{dim}~
\mathrm{SU}({\bf n}) -\mathrm{dim}~ \mathrm{Stab}(X)=
2\displaystyle\sum\limits_{j=1}^rm_j(n_j-m_j)$$ is the dimension of the  $r-$fold tensor product of
 Grassmannians.
\end{proof}
\begin{rem}
$(a)\;$  Let $V\otimes W$ denote  the tensor product
of finite dimensional vector spaces $V$ and $W$, cf.~\cite{g,l} and let
$\displaystyle X\otimes Y : V\otimes W\rightarrow V\otimes W$
be the tensor product of $X\in\mathrm{End}(V)$
and $ Y\in\mathrm{End}(W),$ given by
$\displaystyle v\otimes w\mapsto \displaystyle Xv\otimes Yw,$
for all $v\in V$ and $w\in W.$  Moreover, let $B_V$ and $B_W$ be bases
of $V$ and $W$, respectively. Then, the matrix representation of
$X\otimes Y$ with respect to the product basis
$\{v \otimes w \;|\; v \in B_V,\, w \in B_W\}$ of  $V\otimes W$ is
given by the Kronecker product of the matrix representations of
$A$ and $B$ with respect to $B_V$ and $B_W$.
This clarifies the relation between the ``abstract'' tensor product
of linear maps and the Kronecker product of matrices and justifies
the term ``tensor product'' of Grassmannians when we  refer to
$\mathrm{Gr}^{\otimes r}({\bf m},{\bf n})$.\\
$(b)\;$ It is a well-known fact that the Grassmannian
$\mathrm{Gr}_{m,n}$ is diffeomorphic to the Grassmann manifold
$\mathrm{Grass}_{m,n}$ of all $m-$dimensional subspaces of $\C^n$,
cf.\cite{hm}. Therefore,
$\mathrm{Gr}_{m_1,n_1}\otimes \mathrm{Gr}_{m_2,n_2}$ is diffeomorphic
to \begin{equation}\{V_1\otimes V_2\;|\;
V_1\in\mathrm{Grass}_{m_1,n_1},\; V_2\in\mathrm{Grass}_{m_2,n_2}\}
\subset\mathrm{Grass}_{M,N},
\end{equation}
where $M:=m_1m_2$ and $N:=n_1n_2.$\\[2mm]
%
Both items (a) and (b) readily generalize to an arbitrary
number of Grassmannians.
\end{rem}

\smallskip

We conclude this subsection by pointing out an isometry between the $r-$fold
tensor product of Grassmannians  $\mathrm{Gr}^{\otimes r}({\bf m},{\bf n})$
and the \emph{direct  $r-$fold product of Grassmannians}
\begin{equation}\label{directGrass}
\mathrm{Gr}^{\times r}({\bf m},{\bf n}):=\{(P_1,\dots,P_r)\;|\; P_j\in\mathrm{Gr}_{m_j,n_j},\; j=1,\dots, r\}.
\end{equation}
The vector spaces 
$\mathfrak{her}_N$ and
$\mathfrak{her}_{n_1}\times \dots\times\mathfrak{her}_{n_r}$
endowed with the inner products
\begin{equation}\label{inner1}\langle X,Y \rangle:=\tr(XY)\end{equation}  and
\begin{equation}\label{innerprod3}
\Bigl < (X_1,\dots,X_r), (Y_1,\dots,Y_r) \Bigr > := \tr(X_1Y_1) +\dots+
\tr(X_rY_r),
\end{equation}
induce a Riemannian submanifold structure on $\mathrm{Gr}^{\otimes r}({\bf
  m}, {\bf n})$ and $\mathrm{Gr}^{\times r}({\bf m}, {\bf n})$,  respectively.
\begin{prop}\label{diffeo}
The map
\begin{equation}\label{iso1}
\begin{array}{cc}
\varphi:\mathrm{Gr}^{\times r}({\bf m},{\bf  n})\rightarrow
\mathrm{Gr}^{\otimes r}({\bf m},{\bf  n})\;,\;\;\;
(P_1,\dots,P_r)\mapsto P_1\otimes \dots\otimes P_r
\end{array}
\end{equation}is a diffeomorphism  between $\mathrm{Gr}^{\times r}({\bf
  m},{\bf n})$ and $\mathrm{Gr}^{\otimes r}({\bf m},{\bf n}).$ Moreover, $\varphi$ is a global
Riemannian isometry when the right-hand side of (\ref{innerprod3}) is replaced by  \begin{equation}\label{innerprod4}
 M_1\tr(X_1Y_1) +\dots+ M_r\tr(X_rY_r),
\end{equation}
with $M_j:=\displaystyle\prod\limits_{k=1,\;k\neq j}^rm_k,$ for $j=1,\dots,r.$
\end{prop}

Note  that the isometry between $\mathrm{Gr}^{\times r}({\bf m},{\bf n})$ and
$\mathrm{Gr}^{\otimes r}({\bf m},{\bf n})$ is very special, as in general the
map
\begin{equation}\label{her}
\mathfrak{her}_{n_1}\times \dots\times\mathfrak{her}_{n_r}\rightarrow \mathfrak{her}_N,\;\;
(X_1,\hdots, X_r)\mapsto X_1\otimes\cdots\otimes  X_r
\end{equation} fails even to be injective. For the proof of Proposition $\ref{diffeo}$ we refer to the Appendix.

 \subsection{Geodesics and parallel transport} 
 It is well-known  that every Riemannian manifold $\mathcal{M}$
 carries a unique \textit{Riemannian} or
 \textit{Levi-Civita connection} $\boldsymbol{\nabla}$, e.g.~\cite{ams,hm,udriste}.
 By means of
 $\boldsymbol{\nabla}$, one defines parallel transport and geodesics
 as follows. Let $t \mapsto \mathcal{X}(t)$ be a vector field along a curve
 $\gamma$ on $\mathcal{M}$, i.e.~$\mathcal{X}(t)\in\mathrm{T}_{\gamma(t)}\mathcal{M}$
 for all $t \in \mathbb{R}$. Then, $\mathcal{X}$ is defined to be
 \textit{parallel along} $\gamma$ if  \begin{equation}\label{parallel}
 \boldsymbol{\nabla}_{\dot{\gamma}(t)}\mathcal{X}(t)=0
 \end{equation} for all $t\in \mathbb{R}$. Given $\xi\in\mathrm{T}_{\gamma(0)}\mathcal{M}$,
 there exists a unique parallel vector field $\mathcal{X}$ along $\gamma$
 such that $\mathcal{X}(0)=\xi$ and the vector
 $\mathcal{X}(t) \in\mathrm{T}_{\gamma(t)}\mathcal{M}$
 is called the \emph{parallel transport} of
 $\xi$ to $\mathrm{T}_{\gamma(t)}\mathcal{M}$ along $\gamma.$
 In particular, $\gamma$ is called a \textit{geodesic} on $\mathcal{M}$,
 if $\dot{\gamma}$ is parallel along $\gamma$.


 For the Grassmann manifold  $\mathrm{Gr}_{m,n}$, the curve $t\mapsto\gamma(t)=e^{-t[\xi,P]}Pe^{t[\xi,P]}$  describes the geodesic through $P\in\mathrm{Gr}_{m,n}$ in direction $\xi\in\mathrm{T}_P\mathrm{Gr}_{m,n}$, i.e. $\gamma(t)$ satisfies equation \eqref{parallel} with initial conditions $\gamma(0)=P$ and $\dot{\gamma}(0)=\xi$.
Similarly,  it can be verified that  the parallel transport of  $\eta\in \mathrm{T}_{P}\mathrm{Gr}_{ m, n}$ to $\mathrm{T}_{\gamma(t)}\mathrm{Gr}_{ m, n}$ along the geodesic through $P$ in direction $\xi$ is given by $\eta\mapsto e^{-t[\xi,P]}\eta e^{t[\xi,P]}$. These notions can be straight-forward generalized to the direct product of Grassmannians $\mathrm{Gr}^{\times r}({\bf m},{\bf n})$.

\subsection{The Riemannian gradient and Hessian}
First, let us recall that the \textit{Riemannian gradient}
at $P \in \mathcal{M}$ of a smooth objective function
$f: \mathcal{M} \to \R$ on a Riemannian manifold $\mathcal{M}$
is defined as the unique tangent vector
$\grad f(P)\in \mathrm{T}_P\mathcal{M}$ satisfying
\begin{equation}\label{grad0}
{\rm d}f(P) (\xi)=\langle \grad f(P),\xi\rangle
\end{equation}
$\textrm{for all}\; \xi\in \mathrm{T}_P\mathcal{M}$, where
${\rm d}f(P)$ denotes the differential (tangent map) of $f$ at $P$. 
Moreover, if $\boldsymbol{\nabla}$ is the \emph{Levi-Civita connection} on
$\mathcal{M}$, 
then the \textit{Riemannian Hessian} of $f$ at $P$
is the linear map
${\bf H}_f(P):\mathrm{T}_P\mathcal{M}\rightarrow \mathrm{T}_P\mathcal{M}$
defined by
\begin{equation}\label{subhess}
{\bf H}_f(P)\xi=\boldsymbol{\nabla}_{\xi}\grad f(P),
\end{equation}
for all $\xi\in\mathrm{T}_P\mathcal{M}.$
Now, if  $\mathcal{M}$ is a submanifold of a vector
space $V$, then
\eqref{grad0} and \eqref{subhess} simplify as follows.
Let $\widetilde{f}$ and $\widetilde{\mathcal{X}}$ be smooth extensions of
$f$ and of the vector field $\grad f$, respectively.
Then,
\begin{equation}\label{subgrad}
\grad f(P) =  \pi_P\big(\nabla\widetilde{f}(P)\big),
\quad
{\bf H}_f(P)\xi=\pi_P (D \widetilde{\mathcal{X}} (P)\xi),
\end{equation}
where $\pi_P$ is the orthogonal projection onto
$\mathrm{T}_{ P}\mathcal{M}$ and $\nabla\widetilde{f}$
denotes the standard gradient of $\widetilde{f}$ on $V$.

For the generalized Rayleigh-quotient $\rho_A$ on
$\mathrm{Gr}^{\times r}({\bf m},{\bf n})$, explicit formulas of the
gradient and Hessian will be given in Section 3.3.

\section{\bf The generalized Rayleigh-quotient}\label{pbform}
Let $\mathrm{Gr}^{\otimes r}({\bf m},{\bf n})$ be the $r-$fold tensor
product of Grassmannians  with $({\bf m},{\bf n})$
as in ($\ref{mn}$) and let  $A\in\mathfrak{her}_N,$ $N=n_1n_2\cdots n_r.$
In the following, we analyze the constrained optimization problem
\begin{equation}\label{approx}
\underset{{\bf P}\in \mathrm{Gr}^{\otimes r}({\bf m},{\bf n})}\max
\tr(A {\bf P}),
\end{equation}
which comprises problems from different areas, such as
\emph{ multilinear low-rank approximations of a tensor},
\emph{geometric measures of entanglement},  \emph{subspace clustering} and \emph{combinatorial optimization}.  These applications are naturally stated on a tensor product space. However, for the special case of the Grassmann manifold  they can be reformulated on a direct product space. To this purpose, we define the \emph{generalized Rayleigh-quotient} of
the matrix $A$ as \begin{equation}\label{gRq}
\begin{array}{cc}
\rho_A: \mathrm{Gr}^{\times r}({\bf m},{\bf n})\rightarrow \R, &
\rho_A(P_1,\dots,P_r) :=\displaystyle \tr\Big(A (P_1\otimes\dots\otimes P_r)\Big).
\end{array}
\end{equation}
Based on the isometry between $\mathrm{Gr}^{\otimes r}({\bf m},{\bf n})$ and
$\mathrm{Gr}^{\times r}({\bf m},{\bf n}),$ we can rewrite
problem ($\ref{approx}$) as an optimization task for $\rho_A$
\begin{equation}\label{pb}
 \underset{(P_1,\dots,P_r)\in \mathrm{Gr}^{\times r}({\bf m},{\bf n})}\max
 \rho_A(P_1,\dots,P_r).
\end{equation} In general this is not the case, as we have already pointed out in  \eqref{her}.

The term ``generalized Rayleigh-quotient" is
justified, since for $r=1$ we obtain the classical
Rayleigh-quotient $\rho_A(P)=\tr(AP).$  In the sequel we want to point out some similarities and differences between the generalized  and the classical Rayleigh-quotient.
 It is well known that under the assumption that there is a spectral gap between the eigenvalues of $A\in\mathfrak{her}_N$, there is a unique maximizer and a unique minimizer of the classical Rayleigh-quotient of $A$. Unfortunately, this is no longer the case for the generalized Rayleigh-quotient $\rho_A$. Global maximizers and  global minimizers exist since the generalized Rayleigh-quotient is defined on a compact manifold, but unlike the classical case, it admits also local extrema as the following example shows. For the case when $A$ is of rank one we refer to Example 3 in \cite{lmv2}.
\begin{ex}\label{ex} Let $A=\mathrm{diag}(\lambda_1, \lambda_2, \lambda_3, \lambda_4)\in\mathfrak{her}_{4}$ be a diagonal matrix with $\lambda_2>\lambda_3>\lambda_4>\lambda_1$ and $P_1^*,\; P_2^*\in \mathrm{Gr}_{1,2}$ of the form
\begin{equation}
\begin{array}{cc}
P_1^*=\left[\begin{array}{cc}
1 & 0\\
0 & 0
\end{array}\right] & \textrm{and }\;\; P_2^*=\left[\begin{array}{cc}
0 & 0\\
0 & 1 \end{array}\right].
\end{array}
\end{equation} The maximum of $\rho_A$ is obvious less or equal to $\lambda_2$. Since $\rho_A(P_1^*,P_2^*)=\lambda_2$, we have $(P_1^*,P_2^*)$ as the global maximizer of $\rho_A$. From \eqref{cp} it follows that all $(P_1,P_2)\in\mathrm{Gr}_{1,2}\times \mathrm{Gr}_{1,2}$ with $P_1$ and $P_2$ diagonal, are critical points of $\rho_A$. In particular $(P_2^*,P_1^*)$ is a critical point of $\rho_A$ with  $\rho_A(P_2^*,P_1^*)=\lambda_3<\lambda_2$.  Moreover, one can check by computing the Hessian of $\rho_A$ at $(P_2^*,P_1^*)$ , see \eqref{hessian}, that  $(P_2^*,P_1^*)$ is actually a local maximizer of $\rho_A$. Comparative to the classical Rayleigh-quotient, this strange behavior results from the fact that not all $4\times 4$ permutation matrices are of the form $\Theta_1\otimes\Theta_2$, with $\Theta_1,\;\Theta_2\in\mathrm{SU}_2$.
\end{ex}


While for the classical Rayleigh-quotient one knows that the maximizer and minimizer are orthogonal projectors onto the space spanned by the eigenvectors corresponding to the largest and, respectively, smallest eigenvalues of $A$, it is difficult to provide an analog characterization for the global extrema of the generalized Rayleigh-quotient for an arbitrary matrix $A$.  But, for particular $A$ and $r$ such a characterization is possible.\\ 
{\bf (a)} If $r=2$ and $A$ is of rank one, i.e.  $A=\mathrm{vec}(Y)\mathrm{vec}(Y)^{\dagger}$, with $Y\in\C^{n_1\times n_2}$, then the generalized Rayleigh-quotient can be rewritten as
    \begin{equation}
\rho_A(P_1,P_2)= \tr[\mathrm{vec}(Y)\mathrm{vec}(Y)^{\dagger}(P_1\otimes P_2)]
=\tr(Y^{\dagger}P_1YP_2).
\end{equation}
Under the assumption that $Y$ has full rank and distinct singular values there exist  one maximizer and one minimizer.
The maximizer $(P_1^*,P_2^*)\in\mathrm{Gr}^{\times 2}({\bf m}, {\bf n})$ of $\rho_A$ is given by the  orthogonal projectors onto the space spanned by the $m_*:=\min\{m_1,m_2\}$ left, respective right  singular vectors corresponding to the largest $m_*$ singular values. Similar for the minimizer, the singular vectors corresponding to the smallest $m_*$ singular values.\\
{\bf (b)} If $r$ is arbitrary and  $A$
 diagonalizable   via a transformation of
$\mathrm{SU}({\bf n})=\{\Theta_1\otimes\cdots\otimes\Theta_r\;|\; \Theta_j\in\mathrm{SU}_{n_j}\}$, then
we can assume without loss of generality that $A$ is diagonal.
Moreover, if $A$ can be written as $\Lambda_1\otimes\dots\otimes
\Lambda_r,\;\textrm{with}\; \Lambda_j$ diagonal,  which is always the case when $A=A_1\otimes\cdots\otimes A_r$,  $A_j\in\mathfrak{her}_{n_j}$, then the generalized Rayleigh-quotient becomes a product of $r$ decoupled classical Rayleigh-quotients.
Hence, there is one maximizer and one minimizer.
However, there is a dramatic change if $A$ cannot be written as a Kronecker product of diagonal matrices. In this case $\rho_A$ has also local extrema, as Example \ref{ex} shows.
 From  \eqref{cp} one can immediately formulate the following critical point characterization.
  \begin{prop}\label{prop_diag} Let $A\in\mathfrak{her}_{N}$ be diagonal. Then, $(P_1,\dots,P_r)\in\mathrm{Gr}^{\times r}({\bf m},{\bf n})$ is a critical point of $\rho_A$ if and only if  $P_j$ are permutations of the standard projectors $\Pi_j$, for all $j=1,\dots, r$. 
  \end{prop}

\subsection{Applications}\label{appl}

There is a wide range of applications for problem ($\ref{pb}$) in areas such
as  signal processing, computer vision and quantum information.
We briefly illustrate the broad potential of  ($\ref{pb}$) by four examples.

\subsubsection{\bf Best multilinear rank-$(m_1,\dots,m_r)$ tensor approximation} 
The problem of best approximation of a tensor by a tensor of lower rank
is important in areas such as statistics, signal processing and pattern
recognition. Unlike in the matrix case, there are several rank concepts for a
higher order tensor,  \cite{lmv2, sl}. For the scope of this paper, we focus on  the \emph{multilinear rank} case.

A finite dimensional complex tensor $\mathcal{A}$ of order $r$
is an element of a tensor product  $V_1\otimes\cdots\otimes V_r$,
where $V_1,\dots, V_r$ are complex vector spaces with
$\mathrm{dim}~V_j=n_j.$ Such an element can have various
representations, a common one is the description as an $r-$way
array, i.e.~after a choice of bases for $V_1,\dots, V_r$, the tensor
$\mathcal{A}$ is identified with
$\displaystyle [a_{i_1 \dots  i_r}]_{i_1=1, \dots, i_r=1}^{n_1,
 \dots, n_r}\in\C^{n_1\times n_2\times\dots\times n_r}$, see e.g.~\cite{sl}.
The $j-$th way of  the array is referred to as the $j-$th \emph{mode} of $\mathcal{A}$.
A matrix $X\in\C^{q_j\times n_j}$ acts on a tensor $\mathcal{A}\in\C^{n_1\times n_2\times\dots\times n_r}$ via \emph{mode$-j$ multiplication} $\times_j$, i.e.
\begin{equation}
(\mathcal{A}\times_j X)_{i_1\dots i_{j-1}k_1i_{j+1}\dots i_r}
=\sum\limits_{k_2=1}^{n_j}a_{i_1\dots i_{j-1}k_2i_{j+1}\dots i_r}x_{k_1k_2},
\end{equation} cf.~\cite{lmv1,sl}.

It is always possible to rearrange the elements of $\mathcal{A}$ along
one or, more general, several modes such that they form a
matrix. Let $l_1,\dots, l_q$ and $c_1,\dots, c_p$ be ordered subsets
of $1,\dots, r$ such that
$\{l_1,\dots, l_q\}\cup\{c_1,\dots, c_p\}=\{1,\dots, r\}$.
Moreover, consider the products
$ N_{k}:=n_{l_{k+1}}\cdots n_{l_q},\; N'_{k}:=n_{c_{k+1}}\cdots n_{c_p},$
for $k=0,\dots, q-1$
and $k=0,\dots, p-1$, respectively.
Then, the \emph{matrix unfolding} of $\mathcal{A}$ along
$(l_1,\dots, l_q)$ is a matrix
$A_{(l_1,\dots, l_q)}$ of size $N_{0}\times N'_{0}$
such that the element in position $(i_1,\dots, i_r)$ of
$\mathcal{A}$ moves to position $(s,t)$ in $A_{(l_1,\dots, l_q)}$,
where
\begin{equation}\label{matrix_jk}
s:={i}_{l_q}+\sum\limits_{k=1}^{q-1}({i}_{l_k}-1)N_{k}
\quad\quad\mbox{and}\quad\quad
t := i_{c_p}+\sum\limits_{k=1}^{p-1}(i_{c_k}-1)N'_{k}.
\end{equation}
 As an example, for a third order tensor $\mathcal{A}\in\C^{2\times 2\times 2}$ we obtain the following matrix unfoldings as in \cite{lmv1}
 \begin{equation*}
 \begin{array}{c}
 \begin{array}{cc}
 A_{(1)}=\left[\begin{array}{cccc}
 a_{111} & a_{112} & a_{121} & a_{122}\\
 a_{211} & a_{212} & a_{221} & a_{222}
 \end{array}
 \right], & A_{(2)}=\left[\begin{array}{cccc}
 a_{111} & a_{112} & a_{211} & a_{212}\\
 a_{121} & a_{122} & a_{221} & a_{222}
 \end{array}
 \right] ,
 \end{array}\\[3mm]
 A_{(3)}=\left[\begin{array}{cccc}
 a_{111} & a_{121} & a_{211} & a_{221}\\
 a_{112} & a_{122} & a_{212} & a_{222}
 \end{array}
 \right].
 \end{array}
 \end{equation*}
The \emph{multilinear rank} of $\mathcal{A}\in\C^{n_1\times\dots\times n_r}$
is the $r-$tuple $(m_1,\dots, m_r)$ such that \begin{equation}
m_1=\mathrm{rank}~A_{(1)}\;,\;\; \dots\;\;,\; m_r=\mathrm{rank}~A_{(r)}.
\end{equation} To refer to the multilinear rank of $\mathcal{A}$ we will use
the notation  rank-$(m_1,\dots,m_r)$ or $\mathrm{rank}~
\mathcal{A}=(m_1,\dots,m_r).$
Given a tensor $\mathcal{A}\in\C^{n_1\times \dots\times n_r},$  we are
interested in finding the best rank-$(m_1, \dots, m_r)$ approximation  of $\mathcal{A}$, i.e. \begin{equation}\label{multilinalg}
\underset{\mathrm{rank} (\mathcal{B})\leq(m_1, \dots, m_r)}{\min}\|\mathcal{A}-\mathcal{B}\|.
\end{equation}
Here, $\|\mathcal{A}\|$ is the Frobenius norm of a tensor,
i.e. $\|\mathcal{A}\|^2=\langle \mathcal{A},\mathcal{A}\rangle$ with
\begin{equation}\label{tensor_inn}
\langle\mathcal{A},\mathcal{B}\rangle=\mathrm{vec}(\mathcal{A})^{\dagger}\mathrm{vec}(\mathcal{B})
=\sum\limits_{i_1,\dots,i_r=1}^{n_1,\dots,n_r}\bar{a}_{i_1\dots i_r}b_{i_1\dots i_r}.
\end{equation} Here, $\mathrm{vec}(\mathcal{A})$ refers to the matrix unfolding $A_{(1,\dots,r)}\in\mathbb{C}^{N\times 1}$.
In the matrix case, the solution of the optimization problem
($\ref{multilinalg}$) is given by a truncated SVD, cf.~Eckart-Young
theorem \cite{ey}. However, for the higher-order case, there is no equivalent
of the Eckart-Young theorem. According to the Tucker decomposition \cite{tlr}
or the higher order singular value decomposition (HOSVD) \cite{lmv1}, any
rank-$(m_1,\dots,m_r)$ tensor can be written as a product of a \emph{core}
tensor $\mathcal{S}$ and $r$ Stiefel matrices $X_1\in\C^{m_1\times n_1}, \dots,\;
X_r\in\C^{m_r\times n_r}$ ,
i.e. \begin{equation*}\mathcal{B}=\mathcal{S}\times_1 X_1\times _2
  \dots\times_r X_r,\; \;\;X_j^{\dagger}X_j=I_{m_j},\; j=1,\dots, r.
 \end{equation*}
Thus,  solving   ($\ref{multilinalg}$)  is equivalent to solving the maximization  problem  \begin{equation}
\underset{X_1, \dots, X_r}{\max}\|\mathcal{A}\times_1 X_1\times_2  \dots\times_r X_r\|^2,
\end{equation}
with $ X_j^{\dagger}X_j=I_{m_j},\; j=1,\dots, r$, see e.g.~\cite{es}.  Using  $\mathrm{vec}-$operation and Kronecker product language,
one has \begin{equation}\mathrm{vec}(\mathcal{A}\times_1 X_1\times_2 \dots \times_r
  X_r)=\mathrm{vec}(\mathcal{A})^{\dagger} (X_1\otimes \cdots \otimes X_r).\end{equation} According to \eqref{tensor_inn} and the properties of the \emph{trace} function, the best multilinear rank-$(m_1,\dots,m_r)$ approximation problem becomes
\begin{equation}\label{A_1}
  \underset{(P_1,\dots, P_r)\in \mathrm{Gr}^{\times r}({\bf m},{\bf n})}\max
 \displaystyle \tr\Big(A (P_1\otimes \cdots\otimes P_r)\Big),
 \end{equation}
 with $A=\mathrm{vec}(\mathcal{A})\mathrm{vec}(\mathcal{A})^{\dagger}$ and
 $P_j=X_jX_j^{\dagger},\; j=1,\dots, r.$

\subsubsection{\bf A geometric measure of entanglement}
The task of characterizing and quantifying entanglement is a central
theme in quantum information theory. There exist various ways to
measure the difference between entangled and product states. Here,  we discuss a geometric measure of entanglement,
which is given by the Euclidean distance of $z\in\C^N$ with $\|z\|=1$ to the set of all product states
$\mathcal{P}=\{x_1 \otimes \cdots \otimes x_r\;|\; x_j \in
\C^{n_j},\; \|x_j\| = 1\}$, i.e.
\begin{equation}\label{eq12}
 \delta_{\mathrm{E}}(z) := \underset{x \in
\mathcal{P}}{\min}\|z-x\|^2.
\end{equation}
Since any minimizer of $\delta_E$ is also a maximizer of
 \begin{equation}\label{eq13}
 \underset{x_j \in
\C^{n_j},\; \|x_j\|=1}{\max} |z^{\dagger}(x_1\otimes \cdots\otimes x_r)|,
\end{equation} and vice versa, computing the entanglement measure (\ref{eq12}) is equivalent to solving
\begin{equation}\label{eq14}
  \underset{(P_1, ..., P_r)\in \mathrm{Gr}^{\times r}({\bf m},{\bf n})}\max
 \displaystyle \tr\Big(A (P_1\otimes\cdots\otimes P_r)\Big),
 \end{equation}
with $A=zz^{\dagger}$ and $P_1=x_1x_{1}^{\dagger}, \dots, P_r=x_rx_r^{\dagger}.$
Note that \eqref{eq14} actually constitutes a best rank$-(1,\dots, 1)$
tensor approximation problem \cite{dhk}.

\subsubsection{\bf  Subspace clustering} Subspace segmentation is a fundamental problem in many
applications in computer vision (e.g. image segmentation) and image
processing (e.g. image representation and compression). The problem of clustering data lying
on multiple subspaces of different dimensions can be stated as follows:

{Given a set of data points $X=\{{ x}_l\in\R^{n}\}_{j=1}^L$ which lie approximately in
  $r\geq 1 $ distinct subspaces
$S_k$ of dimension $d_k,\; 1\leq d_k<n$, identify the
subspaces $S_k$ without knowing in advance which points belong to which subspace.}

Every $d_k$ dimensional subspace $S_k\subset\R^n$ can be defined as the kernel of a rank
$m_k=n-d_k$ orthogonal projector $P_k$ of $\R^{n_k},$ with $n_k=n$ as \begin{equation}
S_k=\{x\in\R^n\;|\; P_kx=0\}.
\end{equation}
Therefore, any point $x\in\underset{k=1}{\overset{r}{\cup}} S_k$  satisfies \begin{equation}
\|P_1x\|\cdot\|P_2x\|\cdots\|P_rx\|=0,
\end{equation}
which is equivalent to   \begin{equation}
\tr(xx^{\top}P_1)\tr(xx^{\top}P_2)\cdots\tr(xx^{\top}P_r)=\displaystyle\tr\Big((xx^{\top}\otimes\cdots\otimes
xx^{\top})(P_1\otimes\cdots\otimes P_r)\Big)=0.
\end{equation} Thus, the problem of recovering the subspaces $S_k$ from the data points $X$ can be treated as the following optimization task: 
\begin{equation}\label{clustering}
\begin{array}{cc}
\underset{P\in \mathrm{Gr}^{\times r}({\bf m},{\bf n})}\min \displaystyle\sum\limits_{l=1}^{L} \prod\limits_{k=1}^r \|P_k{ x}_l\|^2
 & =
\underset{P\in \mathrm{Gr}^{\times r}({\bf m},{\bf n})}\min
 \displaystyle \tr\Big(A (P_1\otimes \dots\otimes P_r)\Big),
 \end{array}
\end{equation}
with $P:=(P_1,...,P_r)$ and
\begin{equation}\label{A_sc} A:=\sum\limits_{l=1}^L \underset{r \;\textrm{times}}{\underbrace{{ x}_l{ x}_l^{\top}\otimes \cdots\otimes{ x}_l{ x}_l^{\top}}}.
\end{equation}
We mention that here we have used the same notation $\mathrm{Gr}^{\times r}({\bf m}, {\bf n})$ to refer to the direct $r-$fold product of \textit{real} Grassmannians.

\smallskip

For best multilinear rank tensor approximation and subspace clustering
applications, numerical experiments are presented at the end of Section 4.

\subsubsection{\bf A combinatorial problem}
Let $\Lambda=(\lambda_{jk})_{j=1,k=1}^{n_2,n_1}$ be a given array of positive real numbers and let $m_1\leq n_1,\; m_2\leq n_2$ be fixed. Find $m_1$
 columns and $m_2$ rows such that the sum of the corresponding entries
$\lambda_{jk}$ is maximal, i.e. solve the combinatorial maximization problem
\begin{equation}\label{combinatorical}
\underset{\underset{|J|=m_2}{J\subset\{1,\dots,n_2\}}}\max
  \underset{\underset{|K|=m_1}{K\subset\{1,\dots,n_1\}}}\max\displaystyle\sum\limits_{j\in
    J,\; k\in K}\lambda_{jk}.
\end{equation}
We can permute $m_1$ columns and $m_2$ rows of $\Lambda$ by right and left multiplication with permutations  of the standard projectors $\Pi_1$ and $\Pi_2,$ respectively.  Hence,  problem \eqref{combinatorical} is solved by finding permutation matrices $\sigma_1$ and $\sigma_2$ which maximize:
\begin{equation}\label{combinatorical2}
\sum\limits_{i,j}  ( \Pi_{\sigma_2}\Lambda \Pi_{\sigma_1})_{ij},
\end{equation} where $\sum_{i,j}$ is the sum over all entries and  $\Pi_{\sigma_1}:=\sigma_1^{\top}\Pi_1 \sigma_1,\; \Pi_{\sigma_2}:=\sigma_2^{\top}\Pi_2 \sigma_2$.
The sum in \eqref{combinatorical2} can be written as
\begin{equation}\label{combinatorical3}
\sum\limits_{i,j}  (\Pi_{\sigma_2}\Lambda \Pi_{\sigma_1})_{ij}=\sum \limits_{i,j}\biggl( (\Pi_{\sigma_1}\otimes \Pi_{\sigma_2})\mathrm{vec}(\Lambda )\biggr)_{ij}=\tr\biggl(A (\Pi_{\sigma_1}\otimes \Pi_{\sigma_2})\biggr),
\end{equation} where $A:=\mathrm{diag}(\mathrm{vec}(\Lambda))$. The last equality in \eqref{combinatorical3} holds since $\Pi_{\sigma_1}\otimes \Pi_{\sigma_2}$ is diagonal, too.
 According to Proposition \ref{prop_diag},   we have the following equivalence
 \begin{equation}\label{equivalent}
  \underset{\sigma_1,\;\sigma_2}\max \tr\biggl(A (\Pi_{\sigma_1}\otimes \Pi_{\sigma_2})\biggr)\equiv
 \underset{(P_1,P_2)\in\mathrm{Gr}^{\times 2}({\bf m},{\bf n})}\max \tr\biggl(A (P_1\otimes P_2)\biggr).
 \end{equation} Hence, we can embed the combinatorial maximization problem \eqref{combinatorical} into our continuous optimization task \eqref{pb}.
 The generalization of \eqref{combinatorical} to $\Lambda$ being an arbitrary multi-array is straight-forward.

Problems of this type arise in multi-decision processes such as the following.
Assume that a company has $n_1$ branches and each branch produces $n_2$ goods. If $\lambda_{jk}$ denotes the gain of the $j-$th branch with the $k-$th good, then one could be interested to reduce the number of producers and goods to $m_1$ and $m_2$, respectively, which give maximum benefit.

\subsection{Riemannian optimization}\label{sec_grad}
We continue our investigation of problem ($\ref{pb}$) by computing the
gradient and the Hessian of $\rho_A$.
In the following lemma we establish multilinear maps $\Psi_{A,j}$, which will
enable us to derive clear expressions for the gradient and the Hessian of $\rho_A.$
\begin{lem}\label{lemma1}
Let $A\in\mathfrak{her}_N$  and $(X_1,\dots,
X_r)\in\mathfrak{her}_{n_1}\times\cdots\times\mathfrak{her}_{n_r}.$ Then, for all
$j=1,\dots,r$ there exists a unique map
$\Psi_{A,j}:\mathfrak{her}_{n_1}\times\cdots\times\mathfrak{her}_{n_r}\rightarrow \mathbb{C}^{n_j\times n_j}$
such that
\begin{equation}\label{lema1}
\tr\Big(A(X_1\otimes\dots\otimes X_j Z\otimes\dots\otimes X_r)\Big)
= \tr\Big( \Psi_{A,j}(X_1,\dots, X_r) Z\Big)
\end{equation}
holds for all $Z\in \mathbb{C}^{n_j\times n_j}.$ 
In particular, one has
\begin{equation}\label{delta_simple}
\begin{array}{ll}
\tr\bigg(A(X_1\otimes\dots\otimes
  X_r)\bigg)  & =\tr \bigg(\Psi_{A,1}(I_{n_1}, X_2,\dots, X_r) X_1\bigg)\\[2mm]
   & =\dots
= \tr \bigg(\Psi_{A,r}(X_1,\dots,X_{r-1}, I_{n_r}) X_r\bigg).
\end{array}
\end{equation}
Moreover, for $A := A_1\otimes\cdots\otimes A_r$ the maps
$\Psi_{A,j}$ exhibit the explicit form
\begin{equation}\label{scurtare1}
\Psi_{A,j}(X_1,\dots, X_r) =
\Bigg(\prod\limits_{k=1,\,k\neq j}^r
\tr(A_k X_k)\Bigg)A_jX_j.
\end{equation}
\end{lem}
\begin{proof}
Fix $j$ and consider the linear functional
\begin{equation*}
Z \mapsto \lambda_A(Z) :=
\tr\Big(A(X_1\otimes\dots\otimes X_j Z\otimes\dots\otimes X_r)\Big).
\end{equation*}
By the Riesz representation theorem, there exists a unique
$B_j \in \mathbb{C}^{n_j\times n_j}$ 
 such that
$\lambda_A(Z) = \tr\big(B_j Z\big)$
for all $Z \in \mathbb{C}^{n_j\times n_j}$. 
Therefore, the map $\Psi_{A,j}$ is given by\\ $(X_1,\dots, X_r) \mapsto \Psi_{A,j}(X_1,\dots, X_r) := B_j$.
It is straightforward to show that $\Psi_{A,j}$ is multilinear in
$X_1,\dots, X_r$. Now, choosing
$Z := X_j$ and $X_j:=I_{n_j}$ in \eqref{lema1} immediately yields \eqref{delta_simple}.
Moreover, \eqref{scurtare1} follows from the trace equality
\begin{equation*}
\tr\big(A_1X_1\otimes\dots\otimes A_jX_jZ\otimes\dots\otimes A_rX_r\big)
= \bigg(\prod\limits_{k=1,\,k\neq j}^r \tr(A_k X_k)\bigg)
\tr(A_jX_jZ).
\end{equation*}
Thus the proof of Lemma \ref{lemma1} is complete.
\end{proof}
\begin{rem}
The linear maps $\Psi_{A,j}$ constructed in the above proof are almost
identical to the so-called partial trace operators --- a well-known
concept from multilinear algebra and quantum mechanics \cite{br}.
\end{rem}

Next, we show how to compute $\Psi_{A,j}(X_1,\dots, X_r)$ for given
$(X_1,\dots,X_r)\in\mathfrak{her}_{n_1}\times\cdots\times\mathfrak{her}_{n_r}$
if $A$ is not a pure tensor product $A_1\otimes\cdots\otimes A_r$.
\begin{lem}\label{t_d}
Let $A\in\mathfrak{her}_N$ and
$(X_1,\dots,X_r)\in\mathfrak{her}_{n_1}\times\cdots\times\mathfrak{her}_{n_r}$.
Then, 
the $(s,t)-$position of $\Psi_{A,j}(X_1, \dots, X_r) \in \mathbb{C}^{n_j\times n_j}$
is given by
\begin{equation}\label{tau_delta}
\mathlarger{\mathlarger{\sum}}\limits_{\substack{i_l=1,\, l \neq j \\ l=1,\dots,r}}^{n_l}
\mathrm{e}_{i_1}^\top\otimes\cdots\otimes\mathrm{e}_{s}^\top\otimes\cdots\otimes\mathrm{e}_{i_r}^\top
A (X_1\otimes\cdots\otimes X_r)
\mathrm{e}_{i_1}\otimes\cdots\otimes\mathrm{e}_{t}\otimes\cdots\otimes \mathrm{e}_{i_r},
\end{equation}
where
$\{\mathrm{e}_{i_l}\}_{i_l=1}^{n_l}$ denotes the standard basis of  $\mathbb{C}^{n_l}$.
\end{lem}
\begin{proof}
Let $1\leq s, t\leq n_j$. Then, the element in the $(s,t)$ position of the matrix $\Psi_{A,j}(X_1, \dots, X_r)$ is given by
\begin{equation*}
\begin{array}{ll}
 \mathrm{e}_{s}^{\top}\Psi_{A,j}(X_1, \dots, X_r)\mathrm{e}_{t} & = \tr\biggl(\Psi_{A,j}(X_1, \dots, X_r)\mathrm{e}_{t}\mathrm{e}_{s}^{\top}\biggr)\\[3mm]
 & = \tr \biggl(A(X_1\otimes\cdots\otimes X_j\mathrm{e}_{t}\mathrm{e}_{s}^{\top}\otimes\cdots\otimes X_r ) \biggr)\\[3mm]
 & = \tr\biggl(A (X_1\otimes\cdots\otimes X_r)(I_{n_1}\otimes\cdots\otimes \mathrm{e}_{t}\mathrm{e}_{s}^{\top}\otimes\cdots\otimes I_{n_r} ) \biggr).
\end{array}
\end{equation*}
Hence, \eqref{tau_delta} follows from the identity
$I_{n_l}=\sum\limits_{i_l=1}^{n_l}\mathrm{e}_{i_l}\mathrm{e}_{i_l}^{\top}$.
\end{proof}
\begin{rem} Let $A\in\mathfrak{her}_N$ and
$(X_1,\dots,X_r)\in\mathfrak{her}_{n_1}\times\cdots\times\mathfrak{her}_{n_r}$.
A straightforward consequence of the identity
\begin{equation}\label{id}
\tr\biggl(A (X_1\otimes\cdots \otimes
Z\otimes \cdots\otimes X_r)\biggr)^\dagger
= \tr\biggl(A(X_1\otimes\cdots \otimes
Z^\dagger \otimes \cdots\otimes X_r) \biggr),
\end{equation}  for all $Z\in\C^{n_j\times n_j}$, shows that
$\Psi_{A,j}(X_1,\dots, I_{n_j}, \dots,  X_r)$ is Hermitian.
\end{rem}
For simplicity of writing, whenever
$(P_1,\dots, P_r)\in\mathrm{Gr}^{\times r}({\bf m},{\bf n})$
is understood from the context,  we use the following shortcut
\begin{equation}\label{shortcut}
\widehat{A}_j:=\Psi_{A,j}(P_1,\dots, I_{n_j}, \dots,  P_r).
\end{equation}

Now, we can give an explicit formula for the Riemannian gradient of
$\rho_A$ and derive necessary and sufficient critical point conditions.
\begin{thm}\label{critic}
Let $A\in\mathfrak{her}_ N,\; P:= (P_1,\dots, P_r)\in\mathrm{Gr}^{\times r}({\bf m},{\bf n})$  and let  $\rho_A$ be the generalized Rayleigh-quotient on $\mathrm{Gr}^{\times r}({\bf m},{\bf n}).$
Then, one has the following:

$(i)\;\;\;$ The gradient of $\rho_A$ at $P$ with respect to the Riemannian metric $(\ref{innerprod3})$  is
\begin{equation}\label{grad}
\grad \rho_A(P) = \displaystyle\left ( \mathrm{ad}^2_{P_1}\widehat{A}_1,
\dots,\mathrm{ad}^2_{P_r} \widehat{A}_r\right).
\end{equation}

$(ii)\;\;\;$ The critical points of $\rho_A$ on $\mathrm{Gr}^{\times r}({\bf m},{\bf n})$ are characterized by
\begin{equation}\label{cp}
[P_j, \widehat{A}_j ] = 0
\end{equation}  i.e.  $P_j$, $j=1,\dots,r$ is the orthogonal projector onto an $m_j-$dimensional invariant subspace  of $\widehat{A}_j$.
\end{thm}
\begin{proof}$(i)\;$ Fix $P:=(P_1,\dots, P_r)\in\mathrm{Gr}^{\times r}({\bf m},{\bf n})$  and let $\widetilde{\rho}_A$ denote the canonical smooth extension of $\rho_A$ to $\mathfrak{her}_{n_1}\times\cdots\times\mathfrak{her}_{n_r}.$
Then,
\begin{equation}
D\widetilde{\rho}_A(P)(X)=\sum\limits_{j=1}^r\tr\Big(A(
P_1\otimes\dots\otimes X_j\otimes \dots \otimes
P_r)\Big)=\sum\limits_{j=1}^r \tr(\widehat{A}_jX_j),
\end{equation} for all $X:=(X_1,\dots, X_r)\in \mathfrak{her}_{n_1}\times\cdots\times\mathfrak{her}_{n_r}.$
   From (\ref{innerprod3}), we obtain that  the gradient of $\widetilde{\rho}_A$ at $P$ is given by
  $ \nabla\widetilde{\rho}_A(P) = (\widehat{A}_1,\dots,\widehat{A}_r).$
  Thus,  according to  \eqref{ad} and (\ref{subgrad}),  \begin{equation}
 \grad~ \rho_A(P) = \displaystyle\left ( \mathrm{ad}^2_{P_1}\widehat{A}_{1},
\dots,\mathrm{ad}^2_{P_r} \widehat{A}_{r}\right).
\end{equation}
$(ii)\;$ $P:=(P_1,\dots, P_r)\in\mathrm{Gr}^{\times r}({\bf m},{\bf n})$ is a critical point of
$\rho_A$ iff  $\grad~\rho_A (P)=0$.  This is equivalent to
\begin{equation}\label{interm1}
P_j[P_j,\widehat{A}_j]=\displaystyle [P_j,\widehat{A}_j]P_j,
\end{equation}
 for all $j=1,\dots,r.$ By multiplying \eqref{interm1} once from the left with $P_j$ and once from the right with $P_j$, we obtain that $P_j\widehat{A}_j=P_j\widehat{A}_jP_j$ and $\widehat{A}_jP_j=P_j\widehat{A}_jP_j$. Hence, the conclusion $[P_j,\widehat{A}_j]=0$ holds for all $j=1,\dots, r$.
\end{proof}

\smallskip

As a consequence of  Theorem $\ref{critic}$, we immediately obtain  the following  necessary and sufficient critical point condition.
\begin{cor}\label{critic1}
Let $A\in\mathfrak{her}_ N,\; P:= (P_1,\dots, P_r)\in\mathrm{Gr}^{\times r}({\bf m},{\bf n})$  and let $\Theta_j\in\mathrm{SU}_{n_j}$ be such that $\Theta_j^{\dagger}P_j\Theta_j=\Pi_j,$  where $\Pi_j$ is the standard projector in $\mathrm{Gr}_{m_j,n_j}.$
We write \begin{equation}\label{critic2}
\Theta_j^{\dagger}\widehat{A}_j\Theta_j = \left[
\begin{array}{cc}
 \Psi_{j}' & \Psi_j'''\\[1mm]
 \Psi_j'''^{\dagger} & \Psi_{j}''\\
\end{array}
\right], \end{equation} with $\Psi_{j}'\in\mathfrak{her}_{m_j},\; \Psi_{j}''\in\mathfrak{her}_{n_j-m_j},$ and $\Psi_j'''\in\C^{m_j\times(n_j-m_j)}$. Then
$P$ is a critical point of $\rho_A$ if and only if \begin{equation}
\Psi_j'''=0,                                                                                                                                                                        \end{equation}
for all $j=1,\dots,r.$
\end{cor}

\

For the rest of this section we are concerned with the computation
of the Riemannian Hessian of  $\rho_A$  and also with its non-degeneracy at critical points.

\begin{thm}\label{thhess}
Let $A\in\mathfrak{her}_N$ and $P:=(P_1,\dots,P_r)\in
\mathrm{Gr}^{\times r}({\bf m},{\bf n})$. Then, the Riemannian Hessian of $\rho_A$ at $P$ is the unique self-adjoint operator
\begin{equation}\label{hessian}
\begin{array}{cc}
{\bf H}_{\rho_A}(P):\mathrm{T}_{P}\mathrm{Gr}^{\times r}({\bf m},{\bf n})\rightarrow \mathrm{T}_{P}\mathrm{Gr}^{\times r}({\bf m},{\bf n}),\\[1mm]
\xi:=(\xi_1,\dots,\xi_r)\mapsto
{\bf H}_{\rho_A}(P)(\xi) :=
\Big({\bf H}_{1}(\xi),\dots, {\bf H}_{r}(\xi)\Big),
\end{array}
\end{equation}
defined by
\begin{equation}\label{total}
{\bf H}_j(\xi) := -\mathrm{ad}_{P_j}\mathrm{ad}_{\widehat{A}_j} \xi_j +\sum\limits_{k=1, k\neq j}^r  \mathrm{ad}^2_{P_j} \Psi_{A,j} (P_1,\dots, I_{n_j},\dots, \xi_k, \dots, P_r),
\end{equation}
where $\widehat{A}_j$ is the shortcut for $ \Psi_{A,j} (P_1,\dots, I_{n_j},\dots,  P_r)$.
\end{thm}
\begin{proof}
Let $( \widetilde{\cal X}_1,\dots, \widetilde{\cal X}_r)$ denote a
smooth extension of $\grad\rho_A$ to
$\mathfrak{her}_{n_1}\times\cdots\times\mathfrak{her}_{n_r}$.
According to  (\ref{grad}), we can choose
$P \mapsto \widetilde{\cal X}_j(P)  =\mathrm{ad}_{P_j}^2 \widehat{A}_j.$
 Then,
\begin{equation}
\begin{array}{ll}
D\widetilde{\cal X}_j(P)(X) & = \mathrm{ad}_{X_j}\mathrm{ad}_{P_j}\widehat{A}_j + \mathrm{ad}_{P_j}\mathrm{ad}_{X_j}\widehat{A}_j \\[1mm]   & +\sum\limits_{k=1, k\neq j}^r  \mathrm{ad}^2_{P_j} \Psi_{A,j} (P_1,\dots, I_{n_j},\dots, X_k, \dots, P_r),
 \end{array}
\end{equation}
for all $P := (P_1,\dots, P_r)$ and $X:=(X_1,\dots, X_r)$ in
$\mathfrak{her}_{n_1}\times\cdots\times\mathfrak{her}_{n_r}.$
 Notice that, the derivative of the linear  map $P_k\mapsto \Psi_{A,j}(P_1,\cdots, I_{n_j},\dots, P_k, \dots, P_r)$ in direction $X_k\in\mathfrak{her}_{n_k}$ ($k\neq j$) is $\Psi_{A,j} (P_1,\dots, I_{n_j}, \dots, X_k, \dots, P_r)$.
Applying (\ref{ad}) and (\ref{subgrad}), the Riemannian Hessian of
 $\rho_A$ at $P \in \mathrm{Gr}^{\times r}({\bf m},{\bf n})$  is given by \eqref{hessian} and \eqref{total}.
Here, we have used the following two facts:

\smallskip
\noindent
(i) Clearly, $\mathrm{ad}_{\widehat{A}_j}\xi_j$ is skew-hermitian and hence
\begin{equation}
-\mathrm{ad}_{P_j}\mathrm{ad}_{\widehat{A}_j} \xi_j =
\mathrm{ad}_{P_j}\mathrm{ad}_{\xi_j}\widehat{A}_j
\end{equation}
is in the tangent space $\mathrm{T}_{P_j}\mathrm{Gr}_{m_j,n_j}$
for all
$\xi_j\in\mathrm{T}_{P_j}\mathrm{Gr}_{m_j,n_j}.$

\smallskip
\noindent
(ii) A straightforward computation shows that
$\mathrm{ad}_{\xi_j}\mathrm{ad}_{P_j}\widehat{A}_j$ is in the
orthogonal complement of $\mathrm{T}_{P_j}\mathrm{Gr}_{m_j,n_j}$
and hence
\begin{equation}
\mathrm{ad}^2_{P_j}\mathrm{ad}_{\xi_j}\mathrm{ad}_{P_j}\widehat{A}_j=0
\end{equation}
for all $\xi_j\in\mathrm{T}_{P_j}\mathrm{Gr}_{m_j,n_j}.$
\end{proof}

By restricting the tangent vectors $(\xi_1,\dots,\xi_r)\in\mathrm{T}_P\mathrm{Gr}^{\times r}({\bf m},{\bf n})$ to the vectors of the form $(0,\dots,\xi_j,\dots,0)$, it follows immediately a necessary condition for the non-degeneracy of the Hessian at local extrema.
\begin{thm}
Let $A\in\mathfrak{her}_N,$ and $P\in\mathrm{Gr}^{\times
  r}({\bf m},{\bf n})$ be a local maximizer (local minimizer) of $\rho_A.$
If ${\bf H}_{\rho_A}(P)$ is non-degenerate, then for all
$j=1,\dots,r$ the equality
\begin{equation}\label{suff_hess}
\sigma(\Psi_{j}')\cap\sigma(\Psi_{j}'')=\emptyset,
\end{equation}
holds with $\Psi_{j}'$ and $\Psi_{j}''$ as in $(\ref{critic2})$.
Here, $\sigma(X)$ denotes the spectrum of $X$.
\end{thm}
\begin{rem}
In the case when $A\in\mathfrak{her}_N$ can be diagonalized by elements in $\mathrm{SU}(\bf{n})=\{\Theta_1\otimes\cdots\otimes\Theta_r\;|\;\Theta_j\in\mathrm{SU}_{n_j}\}$, condition\eqref{suff_hess} is  also sufficient for the nondegeneracy of the Hessian of $\rho_A$ at local extrema.
\end{rem}

\smallskip

In the remaining part of the section we derive a genericity statement
concerning the critical points of the generalized Rayleigh-quotient.
The result is a straightforward consequence of the parametric
transversality theorem \cite{hirsch}.
Let $\mathrm{V}$,  $\mathcal{M}$, $\mathcal{N}$ be smooth manifolds and  
$F:\mathrm{V}\times\mathcal{M}\rightarrow\mathcal{N}$
a smooth map.
Moreover, let
$\mathrm{T}_{(A,P)}F:V \times \mathrm{T}_P\mathcal{M}\rightarrow
\mathrm{T}_{F(A,P)}\mathcal{N}$ denote the tangent map of $F$
at $(A,P)\in V\times\mathcal{M}$. We say that $F$ is \emph{transversal}
to a submanifold $S\subset \mathcal{N}$ and write $F\pitchfork S$ if
\begin{equation}\label{transverse}
\Im \mathrm{T}_{(A,P)} F + \mathrm{T}_{F(A,P)} S
=\mathrm{T}_{F(A,P)} \mathcal{N},
\end{equation}
for all $(A,P)\in F^{-1}(S)$. Then, the parametric transversality theorem
states the following.
\begin{thm}(\cite{hirsch})\label{thmtransv}
Let $V,\;\mathcal{M},\;\mathcal{N}$ be smooth manifolds and $S$ a closed
submanifold of $\mathcal{N}$. Let $F:V\times\mathcal{M}\rightarrow\mathcal{N}$
be a smooth map, let $A\in V$ and define $F_A:\mathcal{M}\rightarrow\mathcal{N}$,
$F_A(P):=F(A,P)$. If $F\pitchfork S$, then the set
\begin{equation}\label{residual}
\{A\in V\;|\; F_A\pitchfork S\}
\end{equation}
is open and dense.
\end{thm}

Now, let $f_A:\mathcal{M}\rightarrow\R$ be a smooth function depending
on a parameter $A \in V$ and consider the map
\begin{equation}\label{F*}
F:V \times \mathcal{M}\rightarrow \mathrm{T}^*\mathcal{M},\;\;\;
F(A,P):=\mathrm{d}f_A(P),
\end{equation}
where $\mathrm{T}^*\mathcal{M}$ is the cotangent bundle of $\mathcal{M}$
and $\mathrm{d}f_A(P)$ denotes the differential of $f_A$ at $P \in \mathcal{M}$.
With these notations, our genericity result reads as follows.
\begin{thm}\label{F_trans}
Let $M$, $V$ and $F$ be as above and let $S$ be the image of the zero
section in $\mathrm{T}^*\mathcal{M}$. If $F\pitchfork S$ then for a generic
$A\in\mathrm{V}$ the critical points of the smooth function
$f_A:\mathcal{M}\rightarrow\R$ are non-degenerate.
\end{thm}
\begin{proof}
Fix $A\in V$ and define
\begin{equation}\label{F_A}
F_{A}:\mathcal{M}\rightarrow \mathrm{T}^*\mathcal{M},\;\;\;
F_{A}(P):= F(A,P) 
\end{equation}
From the Transversality Theorem \ref{thmtransv} it follows
that the set
\begin{equation}
 R:=\{A\in V\;|\; F_A\pitchfork S\}
\end{equation}
is open and dense in $V$ if $F\pitchfork S$. In the following, we will
prove that $F_A\pitchfork S$ is equivalent to the fact that the Hessian
of $f_A$ is non-degenerate in the critical points. This will prove the
theorem.

First, notice that $P_c\in F_A^{-1}(S)$ if and only if $P_c\in\mathcal{M}$
is a critical point of $f_A$. Therefore, the transversality condition
\begin{equation}\label{critic_nedeg1}
\begin{array}{c}
\Im\mathrm{T}_{P_c} F_A +\mathrm{T}_{F_A(P_c)} S=
\mathrm{T}_{F_A(P_c)}\mathrm{T}^*\mathcal{M},
\end{array}
\end{equation}
is equivalent to
\begin{equation}\label{critic_nedeg}
\begin{array}{c}
 \Im\mathrm{T}_{P_c} F_A +\mathrm{T}_{0} S=\mathrm{T}_{0}\mathrm{T}^*\mathcal{M}.
\end{array}
\end{equation}
To rewrite this condition \eqref{critic_nedeg} in local coordinates,
let $\varphi:U \rightarrow W \subset\mathrm{T}_{P_c}\mathcal{M}$ be a
chart on an open subset $U \subset\mathcal{M}$ around $P_c$ such that
$\varphi^{-1}(0)=P_c$ and  $D\varphi^{-1}(0)=\mathrm{id}$. Then define
\begin{equation}
\begin{array}{l}
\widetilde{f}_A:=f_A\circ\varphi^{-1}:W\rightarrow\R.
\end{array}
\end{equation}
Moreover, $\varphi$ induces a chart
$\psi:\pi^{-1}(U) \rightarrow
W \times \mathrm{T}^*_{P_c}\mathcal{M} \subset
\mathrm{T}_{P_c}\mathcal{M} \times \mathrm{T}^*_{P_c}\mathcal{M}$
around $F_A(P_c)$ via
\begin{equation}
\begin{array}{c}
\psi(\gamma)=(x,(D\varphi^{-1}(x))^*(\gamma)),
\quad x:=\varphi\circ\pi(\gamma),
\end{array}
\end{equation}
Here, $\pi:\mathrm{T}^*\mathcal{M}\rightarrow\mathcal{M}$ refers to
the natural projection and
$(D\varphi^{-1}(x))^*(\gamma):= \gamma \circ D\varphi^{-1}(x)$.
Thus, for
\begin{equation}
\widetilde{F}_A:=\psi\circ F_A\circ\varphi^{-1}:
W \rightarrow W \times \mathrm{T}^*_{P_c}\mathcal{M}
\end{equation}
one has $\widetilde{F}_A(x)=(x,d\widetilde{f}_A(x))$.
Since transversality of $F_A$ to $S$ is preserved in local
coordinates, \eqref{critic_nedeg} is equivalent to
\begin{equation}\label{transv_local}
\Im D\widetilde{F}_A(0) + \mathrm{T}_{P_c}\mathcal{M}\times \{0\}=\mathrm{T}_{P_c}\mathcal{M}\times\mathrm{T}_{P_c}^*\mathcal{M}.
\end{equation}
Then $D\widetilde{F}_A(0)=(\mathrm{id},d^2\widetilde{f}_A(0))$ yields
that \eqref{transv_local} is fulfilled if and only if
$d^2\widetilde{f}_A(0)$ is nonsingular. Finally, the conclusion
follows form the identity $\mathrm{Hess}_{f_A}(P_c)=d^2\widetilde{f}_A(0)$
which is satisfied due to the fact that $P_c$ is a critical point
and $D\varphi^{-1}(0)=\mathrm{id}$. Here, $\mathrm{Hess}_{f_A}(P_c)$
denotes the Hessian form corresponding to the Hessian
operator via
$\mathrm{Hess}_{f_A}(P_c)(x,y)=\langle {\bf H}_{f_A}(P_c)x,y \rangle$
for all $x,y\in\mathrm{T}_{P_c}\mathcal{M}$.
\end{proof}

For the generalized Rayleigh-quotient, we obtain the following result.
\begin{cor}\label{cor_generic}
The critical points of the generalized Rayleigh-quotient are
generically non-degenerate.
\end{cor}
\begin{proof}
Set $\mathcal{M}:=\mathrm{Gr}^{\times r}({\bf m},{\bf n})$,
$V:=\mathfrak{her}_N$. For the simplicity, we will identify the
cotangent bundle $\mathrm{T}^*\mathcal{M}$ with the tangent bundle
$\mathrm{T}\mathcal{M}$ and work with the map
\begin{equation}\label{F}
F:\mathrm{V}\times\mathcal{M}\rightarrow\mathrm{T}\mathcal{M},
\quad (A,P)\mapsto \mathrm{grad}\rho_A(P),
\end{equation}
instead of \eqref{F*}, where $\mathrm{grad}\rho_A(P)$ is the
Riemannian gradient of $\rho_A$ at $P$. We will show that
$F\pitchfork S$, where $S$ is now the image of the zero section
in $\mathrm{T}\mathcal{M}$, i.e.
\begin{equation}\label{transverse_1}
\Im \mathrm{T}_{(A,P)} F+\mathrm{T}_{F(A,P)} S
= \mathrm{T}_{F(A,P)}\mathrm{T}\mathcal{M},
\end{equation}
for all $(A,P)\in V\times \mathcal{M}$ with $\mathrm{grad}\rho_A(P)=0$.
As in the proof of Theorem \ref{F_trans}, we rewrite the
transversality condition \eqref{transverse_1} in local coordinates,
i.e.
\begin{equation}\label{transv_local_1}
\Im D\widetilde{F}(A,0) + \mathrm{T}_{P}\mathcal{M}\times \{0\}
=\mathrm{T}_{P}\mathcal{M}\times\mathrm{T}_{P}\mathcal{M},
\end{equation}
where
\begin{equation}
\widetilde{F}:=\psi\circ F\circ(\mathrm{id}\times\varphi^{-1}):
V \times W \rightarrow W \times \mathrm{T}_{P}\mathcal{M}.
\end{equation}
Here, $\varphi:U\rightarrow W \subset\mathrm{T}_{P}\mathcal{M}$
is a chart around $P$ with $\varphi^{-1}(0)=P$ and
$D\varphi^{-1}(0)=\mathrm{id}$ and
$\psi: \pi^{-1}(U) \rightarrow W \times  \mathrm{T}_{P}\mathcal{M}
\subset  \mathrm{T}_{P}\mathcal{M}\times \mathrm{T}_{P}\mathcal{M}$
is the corresponding induced chart around $F(A,P)$. With this
choice of charts, we obtain
\begin{equation}
\widetilde{F}(A,x)=(x,\nabla\widetilde{\rho}_A(x)),
\end{equation}
where $\widetilde{\rho}_A:=\rho_A\circ\varphi^{-1}:W \rightarrow \R$.
Since $A\mapsto \widetilde{\rho}_A(0)$ is linear, one has
\begin{equation}
D\widetilde{F}(A,0)(X,\xi)=
\left(\xi,\nabla\widetilde{\rho}_X(0)+d^2\widetilde{f}_A(0)\xi\right).
\end{equation}
Thus, condition \eqref{transv_local_1} holds if and only if
\begin{equation}
\label{transv_local_2}
\Im \nabla\widetilde{\rho}_{(\cdot)}(0)+
\Im d^2\widetilde{f}_A(0) = \mathrm{T}_P\mathcal{M}.
\end{equation}
Finally, we will show that 
$\Im \nabla\widetilde{\rho}_{(\cdot)}(0)=\mathrm{T}_P\mathcal{M}$
which clearly guarantees \eqref{transv_local_2}.
Let $\xi:=(\xi_1,\dots,\xi_r)\in
\big(\Im \nabla\widetilde{\rho}_{(\cdot)}(0)\big)^{\perp}$,
then we obtain
\begin{equation}
\begin{split}
0 & = \langle \nabla\widetilde{\rho}_{X}(0), \xi\rangle
= d\widetilde{\rho}_X(0)\xi = d\rho_X(P)\xi\\[2mm]
& =\tr\biggl(X(\sum\limits_{j=1}^r P_1\otimes \cdots\otimes \xi_j\otimes \cdots\otimes P_r)\biggr),
\end{split}
\end{equation}
for all $X\in\mathfrak{her}_N$. Notice, that the equality
$d\widetilde{\rho}_X(0)\xi = d\rho_X(P)\xi$ follows from
$D\varphi^{-1}(0)=\mathrm{id}$. Therefore,
\begin{equation}\label{sum}
\sum\limits_{j=1}^r P_1\otimes \cdots\otimes \xi_j\otimes \cdots\otimes P_r=0
\end{equation} and this holds if and only if $\xi_1=0,\dots, \xi_r=0$,
since alls summands in \eqref{sum} are orthogonal to each other.
Thus, we have proved that
$\widetilde{F}\pitchfork\mathrm{T}_P\mathcal{M}\times\{0\}$ and hence
$F\pitchfork S$. From the Theorem \ref{F_trans} it follows
immediately that the critical points of the generalized
Rayleigh-quotient are generically non-degenerate.
\end{proof}

Unfortunately, for best multilinear rank tensor approximation and
subspace clustering problems, we cannot conclude from
Corollary \ref{cor_generic} that the critical points of $\rho_A$
are generically non-degenerate. In these cases, the resulting
matrices $A$ are restricted to a thin subset of $\mathfrak{her}_N$
and thus the genericity statement with respect $\mathfrak{her}_N$
in Corollary \ref{cor_generic} does not carry over straight-forwardly.

\

\

\section{Numerical Methods}\label{sec4}

Exploiting the geometrical structure of the constraint set $\mathrm{Gr}^{\times r}({\bf m},{\bf n}),$ we develop two numerical methods, a Newton-like and a conjugated gradient algorithm,  for optimizing  the generalized Rayleigh-quotient $\rho_A$,
with $A\in\mathfrak{her}_N,\; N:=n_1n_2\cdots n_r.$

\subsection{Newton-like algorithm}

The intrinsic  Riemannian Newton algorithm is
described by means of
the Levi-Civita connection taking iteration steps along
geodesics  \cite{gd, st}.
Sometimes geodesics are difficult to determine, thus, here we are
interested in a more general approach, which introduces the Newton
iteration via local coordinates, see~\cite{ams,hht,shub}. More
precisely, we follow the ideas in \cite{hht} and use a pair of local
coordinates on $\mathrm{Gr}^{\times r}({\bf m},{\bf n})$, i.e. normal
coordinates  and QR-coordinates. 

Recall that, a \textit{local parametrization}\footnote{Clearly, one can define a local parametrization more
 generally, i.e.~without requiring the second part of \eqref{paramoplus}.} of $\mathrm{Gr}^{\times r}({\bf m},{\bf
  n})$ around a point $P:=(P_1,\dots, P_r)$ is a smooth map
\begin{equation*}\mu_P:\mathrm{T}_P \mathrm{Gr}^{\times r}({\bf m},{\bf n})\rightarrow \mathrm{Gr}^{\times r}({\bf m},{\bf n})
\end{equation*} satisfying the additional conditions \begin{equation}\label{paramoplus}
\mu_{ P}(0)={ 
P}\;\;\;\;\textrm{and}\;\;\;\; D\mu_{ P}(0)=\mathrm{id}_{\mathrm{T}_P\mathrm{Gr}^{\times r}({\bf m},{\bf n})}.
\end{equation}
\emph {Riemannian normal coordinates} are given by the Riemannian exponential map \begin{equation}\label{exp1}
\begin{array}{l}
\mu_{ P}^{\mathrm{exp}}(\xi)=\bigl(e^{-[\xi_1,P_1]}P_1e^{[\xi_1,P_1]}, \dots, e^{-[\xi_r,P_r]}P_re^{[\xi_r,P_r]}\bigr),\end{array}\end{equation}
while \emph{ QR-type coordinates} are defined by the QR-approximation of the matrix exponential, i.e. \begin{equation}\label{QRcoord1}
\begin{array}{l}
\mu_{ P}^{\mathrm{QR}}(\xi)=
\bigl([X_1]_Q^{\dagger}\;P_1\;[X_1]_Q, \dots, [X_r]_Q^{\dagger}\;P_r\;[X_r]_Q\bigr).
\end{array}\end{equation}
Here $[X_j]_Q$  denotes the $Q-$factor from the unique $QR$
 decomposition of $X_j:= I+[\xi_j,P_j].$

Now, let $P^*:=(P_1^*,\dots,P_r^*)\in\mathrm{Gr}^{\times r}({\bf
  m},{\bf n})$ be a critical point of  $\rho_A$. Choose $P\in\mathrm{Gr}^{\times r}({\bf m},{\bf n})$ in a neighborhood of $P^*$ and perform the following Newton-like iteration
\begin{equation}\label{qNit}
P^{\mathrm{new}}:=\mu_P^{\mathrm{QR}}(\xi),
\end{equation}
where $\xi:=(\xi_1,\dots, \xi_r)\in \mathrm{T}_P\mathrm{Gr}^{\times r}({\bf m},{\bf n})$ is a solution of the Newton equation
\begin{equation}\label{Neq}
{\bf H}_{\rho_A}(P)\xi=-\grad~\rho_A (P).
\end{equation}
Replacing the objects in ($\ref{Neq}$)  by their explicit form computed in the previous section, we get the following Newton equation:
 \begin{equation}\label{sylvester1}
-\mathrm{ad}_{P_j}\mathrm{ad}_{\widehat{A}_j} \xi_j +\sum\limits_{k=1, k\neq j}^r  \mathrm{ad}^2_{P_j} \Psi_{A,j} (P_1,\dots, I_{n_j},\dots, \xi_k, \dots, P_r)=- \mathrm{ad}^2_{P_j}\widehat{A}_j,
\end{equation} for all $j=1,\dots,r.$ As mentioned before, let
$\widehat{A}_j:=\Psi_{A,j}(P_1,\dots,I_{n_j},\dots,P_r)$.
Solving this system in the embedding space $\mathfrak{her}_{n_1}\times\cdots\times\mathfrak{her}_{n_r}$ requires a higher number of parameters than necessary. 
However, exploiting the particular structure of the tangent vectors \begin{equation}\label{xi}
\xi_j = \Theta_j\zeta_{j}\Theta_j^{\dagger} = \Theta_j\left[\begin{array}{cc} 0 & Z_j\\
Z_j^{\dagger} & 0\end{array} \right]\Theta_j^{\dagger},
\end{equation}
allows us to solve (\ref{sylvester1}) with the minimum number of parameters equal to the dimension of $\mathrm{Gr}^{\times r}({\bf m}, {\bf n})$. 
Thus, by multiplying ($\ref{sylvester1}$) from the left with $\Theta_j$ and from the right with $\Theta_j^{\dagger},$ we obtain an equation in the variables $Z_j\in\C^{m_j\times (n_j-m_j)},$ i.e.
\begin{equation}\label{sylvester2}
\Psi_{j}' Z_j - Z_j \Psi_{j}'' - \sum\limits_{k=1,k\neq j}^r \Phi_j(Z_k) = \Psi_{j}''',
\end{equation}
where the terms $\Psi_{j}',\;\Psi_{j}'',\;\Psi_{j}'''$ and $ \Phi_j(Z_k)$ are computed in the following.  Let
\begin{equation}\label{theta}
\Theta_j=\left[\begin{array}{cc}U_j & V_j
 \end{array}
\right]\in\mathrm{SU}_{n_j},
\end{equation}  where $U_j$ and $V_j$ are $n_j\times m_j$ and $n_j\times (n_j-m_j)$ matrices, respectively. Then,
\begin{equation}\label{not1}
\begin{array}{ccc}
\Psi_{j}'=U_j^{\dagger}\widehat{A}_j U_j, & \Psi_{j}''=V_j^{\dagger} \widehat{A}_j V_j, & \Psi_{j}'''= U_j^{\dagger}\widehat{A}_j V_j.
\end{array}
\end{equation}
For expressing $\Phi_j(Z_k)$ with $j<k$ , we
 introduce the multilinear operators
$\Psi_{A,j,k}:\mathfrak{her}_{n_1}\times\cdots\times\mathfrak{her}_{n_r}
\rightarrow\mathbb{C}^{n_j\cdot n_k\times n_j\cdot n_k}$
defined in a similar way as $\Psi_{A,j}$ by
\begin{equation}\label{Psi_Ajk}
\tr\Big(A(X_1\otimes\cdots\otimes X_jS\otimes\cdots\otimes X_kT\otimes\cdots\otimes X_r)\Big) = \tr\Big (\Psi_{A,j,k}(X_1,\dots,X_r)(S\otimes T)\Big),
\end{equation}
for all $S\in\mathbb{C}^{n_j\times n_j}$ and $ T\in\mathbb{C}^{n_k\times n_k}.$
For convenience, we will use the following shortcut
\begin{equation}\label{shortcut2}
\widehat{A}_{jk}:=\Psi_{A,j,k}(P_1,\dots, I_{n_j},\dots, I_{n_k},\dots, P_r)\in\mathfrak{her}_{n_j\cdot n_k} .
\end{equation}
\begin{table}
\fbox{\parbox{12.7cm}{
\begin{minipage}{12.7cm}\vspace{0.25cm}
ALGORITHM 1.\;\;\;\;\;\;\;\;\;\;\;\;\;\;\;\;\;\;\;\; {\bf N-like algorithm}\\
\HRule \vspace{0.25cm}

{\bf Step 1.} {\bf Starting point:} Given $P=(P_1,\dots, P_r)\in\mathrm{Gr}^{\times r}({\bf m},{\bf n})$ choose
\begin{equation*}
\Theta_j=\left[\begin{array}{cc}U_j &  V_j
 \end{array}
\right]\in\mathrm{SU}_{n_j},\; U_j^{\dagger}U_j=I_{m_j},\; V_j^{\dagger}V_j=I_{n_j-m_j},\end{equation*}
such that $P_j=\Theta_j\Pi_j\Theta_j^{\dagger}$,  for $j=1,\dots, r.$\\[1mm]
{\bf Step 2.} {\bf Stopping criterion:}
$\|\grad_{\rho_A}(P)\|/\rho_A(P)<\varepsilon$.\\[1mm]
{\bf Step 3.} {\bf Newton direction:}
Set $$\widehat{A}_j:=\Psi_{A,j}(P_1,\dots,I_{n_j},\dots,P_r)
$$ and compute
$\Psi_{j}',\;\Psi_{j}'',\;\Psi_{j}'''\;\textrm{ as in (\ref{not1})},$  for $j=1,\dots, r.$ \\
Set \begin{equation*}
\widehat{A}_{jk}:=\Psi_{A,j,k}(P_1,\dots, I_{n_j},\dots, I_{n_k},\dots, P_r)
\end{equation*} and compute
$\Phi_{j}(Z_k)$ as in \eqref{a_st} and \eqref{Phi_Zk}, for $j,\;k=1,\dots,r,$ with $ j<k$.
 Furthermore,  set $\Phi_k(Z_j)=\Phi_j(Z_k)^{\dagger}$ and 
solve the Newton equation
\begin{equation}\label{sylvester3}
\Psi_{j}' Z_j - Z_j \Psi_{j}'' - \sum\limits_{k=1,k\neq j}^r \Phi_j(Z_k) = \Psi_{j}''',
\end{equation}
 to obtain $Z_j\in\C^{m_j\times (n_j-m_j)},$ for $j=1,\dots, r.$\\[1mm]
{\bf Step 4.} {\bf QR-updates:}
\begin{equation}\label{up}
\begin{array}{cc}
\Theta_j^{\mathrm{new}}=\Theta_j\left[%
\begin{array}{cc}
  I_{m_j} & -Z_j \\[1mm]
  Z_j^{\dagger} & I_{n_j-m_j}
\end{array}\right]_Q
& \textrm{and}\;\;\;
P_j^{\mathrm{new}}=\Theta_j\Pi_j\Theta_j^{\mathrm{new}^{\dagger}}
\end{array}
\end{equation}
for all $j=1,\dots,r$. Here $[\;]_Q$ refers to the $Q$ part from the QR factorization.\\[1mm]
{\bf Step 5.}
Set $P:=P^{\mathrm{new}}$, $\Theta:=\Theta^{\mathrm{new}}$ and go to Step 2.\end{minipage}\vspace{0.25cm}}}
\label{tabel_N}
\end{table}
Furthermore, we partition the matrix $\widehat{A}_{jk}$ into block form \begin{equation}\label{a_st}
\widehat{A}_{jk}=:[\widehat{\bf a}_{st}]_{s,t=1}^{n_j},
\end{equation} where each $\widehat{\bf a}_{st}$ is an $n_k\times n_k$ matrix.
Then,  the linear map $\Phi_j:\mathbb{C}^{m_k\times (n_k-m_k)}\rightarrow\mathbb{C}^{m_j\times (n_j-m_j)}$ is given by
\begin{equation}\label{Phi_Zk}
Z_k\mapsto \Phi_j(Z_k) := U_j^{\dagger}\Big[
\tr(U_k^{\dagger}\widehat{\bf a}_{st}V_k Z_k^{\dagger}+Z_kV_k^{\dagger}\widehat{\bf a}_{st}^{\dagger}U_k) \Big]_{s,t=1}^{n_j} V_j.
\end{equation}
Finally, the complete Newton-like algorithm for the optimization of  $\rho_A$ on $\mathrm{Gr}^{\times r}({\bf m},{\bf n})$ is given by Algorithm 1.

\

{\bf Suggestions for implementation.} (a)
For an arbitrary matrix $A\in\mathfrak{her}_N$, the computation of $\widehat{A}_j$ and $\widehat{A}_{jk}$ 
is performed according to formula (\ref{tau_delta}). This can be simplified in the case of the applications described in Section 3.3.\\
 {\bf Case 1.} If $A=vv^{\dagger},$ with $v=\mathrm{vec}(\mathcal{A}),$ $\mathcal{A}\in\C^{n_1\times \dots\times n_r}$, then \begin{equation}\label{Psi_tensor}
\widehat{A}_j=B_{(j)}\cdot B_{(j)}^{\dagger}\in\mathfrak{her}_{n_j}\;\;\;\textrm{and}\;\;\;  \widehat{A}_{jk}=C_{(j,k)}\cdot C_{(j,k)}^{\dagger}\in\mathfrak{her}_{n_jn_k},
\end{equation}
where $B_{(j)}$ and $C_{(j,k)}$ are the $j-$th mode and respectively $(j,k)-$th mode  matrices of the tensors   $\mathcal{B}=\mathcal{A}\times_1 U_1^{\dagger}\times_2 \dots\times_k I_{n_k}\times_{k+1} \dots\times_r U_r^{\dagger}$ and $\mathcal{C}=\mathcal{A}\times_1 U_1^{\dagger}\times_2 \dots\times_j I_{n_j}\times_{j+1}\dots\times_k I_{n_k}\times_{k+1} \dots\times_r U_r^{\dagger}$, respectively.\\
{\bf Case 2.} If  $A=\sum\limits_{l=1}^L \underset{r\;\textrm{times}}{\underbrace{x_lx_l^{\dagger}\otimes\cdots\otimes x_lx_l^{\dagger}}}$  with  $x_l\in\mathbb{C}^{n}$ and  $L\in\mathbb{N}$,  then
 \begin{equation}\label{Psi_sc}
\widehat{A}_j=\sum\limits_{l=1}^L\Big(\prod\limits_{\substack{i=1\\i\neq j}}^r \|P_ix_l\|^2\Big) x_lx_l^{\dagger}\;\;\;\textrm{and}\;\;\; \widehat{A}_{jk}=\sum\limits_{l=1}^L\Big(\prod\limits_{\substack{i=1\\ i\neq j, i\neq k}}^r \|P_ix_l\|^2\Big) x_lx_l^{\dagger}\otimes x_lx_l^{\dagger}.
\end{equation}
(b) To solve the system (\ref{sylvester2}), one can rewrite it as a linear equation on $\mathbb{R}^{d}$ ($d$ is the dimension of $\mathrm{Gr}^{\times r}({\bf m},{\bf n})$) using matrix Kronecker products and $\mathrm{vec}-$operations, then solve this by any linear equation solver.\\[1mm]
(c)
The computation of geodesics on matrix manifolds
usually requires the matrix exponential map, which is in general
an expensive procedure of order $O(n^3)$.
Yet, for the particular case of the
Grassmann manifold $\mathrm{Gr}_{m,n}$, Gallivan et.al.~\cite{gsld}
have developed an efficient method to compute the matrix
exponential, reducing the complexity order to $O(nm^2)$ ($m<n$).
Our approach, however, is based on a first order approximation
of the matrix exponential $\mathrm{e}^{[\zeta,\Pi]}$
followed by a QR-decomposition to preserve orthogonality/unitarity.
Explicitly, it is given by
\begin{equation}
 \left[\begin{array}{cc}
  I_{m} & -Z \\
  Z^{\dagger} & I_{n-m}
\end{array}\right ]_Q=W\left[
\begin{array}{ccc}
 D^{-1} & \Sigma D^{-1} & 0\\
-\Sigma^{\dagger}D^{-1} & D^{-1} & 0\\
0 & 0 & I_{n-2m}
\end{array}\right]
W^{\dagger},
\end{equation} where $Z=X\Sigma Y^{\dagger}$ with $X\in \mathrm{SU}_{m},\; Y\in\mathbb{C}^{(n-m)\times m}$, $Y_j^{\dagger}Y_j=I_{m_j}$ and $\Sigma\in\mathbb{C}^{m\times m}$ diagonal. Furthermore,
\begin{equation}
\begin{array}{cc}
W:=\left[\begin{array}{ccc}
  X^{\dagger} & 0 & 0\\
  0 & Y & Y'
\end{array}\right ]\in\mathrm{SU}_{n}, &
D:=\sqrt{I_{m}+\Sigma^{\dagger}\Sigma},
\end{array}
\end{equation} where $[Y\;\; Y']\in\mathrm{SU}_{n-m}$ is an unitary completion of $Y$. The computational complexity of this QR-factorization is of order $O((n-m)m^2)$.\\[1mm]
(d) The convergence of the algorithm is not guaranteed for arbitrary initial conditions $P\in\mathrm{Gr}^{\times r}({\bf m},{\bf n})$ and even in the case of convergence the limiting point need not be a local maximizer of the function. To overcome this, one could for example test if the computed direction is ascending, else take the gradient as the new direction. Furthermore, one can make an iterative line-search in the ascending direction.

\smallskip

In the following theorem we prove that the sequence generated by  Algorithm 1 converges quadratically to a critical point of the generalized Rayleigh-quotient $\rho_A$ if  the sequence starts in a sufficiently small neighborhood of the critical point.
\begin{thm} Let $A\in\mathfrak{her}_N$ and $P^*\in
\mathrm{Gr}^{\times r}({\bf m},{\bf n})$ be a non-degenerate critical point of the generalized Rayleigh-quotient $\rho_A$,
  then the sequence generated by the N-like  algorithm  converges locally quadratically to  $P^*$.
\end{thm}
\begin{proof} For the critical point $P^*\in \mathrm{Gr}^{\times r}({\bf m},{\bf n})$, the Riemannian coordinates ($\ref{exp1}$) and the QR- coordinates
  ($\ref{QRcoord1}$) satisfy the condition
  $D\mu_{P^*}^{\mathrm{exp}}(0)=D\mu_{P^*}^{\mathrm{QR}}(0)=\mathrm{id}_{\mathrm{T}_{P^*}\mathrm{Gr}^{\times r}({\bf m}, {\bf n})}$. Thus, according to Theorem 4.1. from \cite{hht} there exists a neighborhood $V\subset
   \mathrm{Gr}^{\times r}({\bf m},{\bf n})$ such that the sequence of  iterates  generated by the N-like algorithm converges quadratically to $P^*$ when the initial point $P$ is in  $V$.
\end{proof}

\subsection{Riemannian conjugated gradient algorithm}

The quadratic convergence of the Newton-like algorithm has the drawback of high computational complexity. Solving the Newton equation (\ref{sylvester2}) yields a cost per iteration of order $O(d^3)$, where $d$ is the dimension of $\mathrm{Gr}^{\times r}({\bf m},{\bf n}).$
In what follows, we offer
as an alternative to reduce the computational costs of the Newton-like algorithm by a conjugated gradient method. The linear conjugated gradient (LCG) method is used for solving large systems of linear equations with a symmetric positive definite matrix, which is achieved by iteratively  minimizing a convex quadratic function  $x^{\dagger}Ax$. The initial direction $d_0$ is chosen as the steepest descent and every forthcoming direction $d_j$ is required to be conjugated  to all the previous ones, i.e. $d_j^{\dagger} A d_k=0,$ for all $k=0,\cdots,j-1$. The exact maximum along a direction gives the next iterate. Hence, the optimal solution is found in at most $n$ steps, where $n$ is the dimension of the problem.
 Nonlinear conjugated gradient (NCG)  methods use the same approach for general functions $f:\R^n\rightarrow\R$, not necessarily convex and quadratic. The  update in this case  reads as
\begin{equation*}
\begin{array}{cc}
x^{\mathrm{new}}=x+\alpha d & \mathrm{and}\;\;\;
d^{\mathrm{new}}=-\nabla f(x^{\mathrm{new}}) +\beta d,
\end{array}
\end{equation*}
where the step-size $\alpha$ is obtained by a line search in the direction $d$
\begin{equation}\label{stepsize}
\alpha=\underset{t}{\mathrm{arg}\min} f(x+t d)
\end{equation} and  $\beta$ is given by one of the
formulas: Fletcher-Reeves, Polak-Ribiere,
Hestenes-Stiefel, or other.
We refer to \cite{st} for the generalization of the NCG  method
to a Riemannian manifold.
For the computation of the step-size along the geodesic in
direction $\xi$,  an exact line search --- as in the classical
case --- is an extremely expensive procedure. Therefore, one
commonly approximates (\ref{stepsize}) by an Armijo-rule,
which ensures at least that the step
length decreases the function sufficiently.
We, however, have decided to compute the step-size
by performing a one-dimensional Newton-step along the geodesic,
since in the neighborhood of a critical point one Newton step can
lead very close to the solution.
Therefore, at
$P=(P_1,\cdots,P_r)\in\mathrm{Gr}^{\times r}({\bf m},{\bf n})$ the  step-size in direction $\xi=(\xi_1,\cdots,\xi_r)\in\mathrm{T}_P\mathrm{Gr}^{\times r}({\bf m}, {\bf n})$  is given by \begin{equation}
\alpha=-\frac{(\rho_A\circ\gamma)'(0)}{(\rho_A\circ\gamma)''(0)}
\end{equation}
where $\gamma:I\rightarrow\mathrm{Gr}^{\times r}({\bf m},{\bf n})$ is the
unique geodesic through $P$ in direction $\xi$.

\begin{table}[h!]
\fbox{\parbox{12.7cm}{
\begin{minipage}{12.7cm}\vspace{0.25cm}
ALGORITHM 2.\;\;\;\;\;\;\;\;\;\;\;\;\;\;\;\;\;\;\;\;\;\;\;\;\;\; {\bf RCG algorithm}\\
\HRule \vspace{0.25cm}

{\bf Step 1.}
 {\bf Starting point:} Given $P=(P_1,\dots, P_r)\in\mathrm{Gr}^{\times r}({\bf m},{\bf n})$ choose
\begin{equation*}
\Theta_j=\left[\begin{array}{cc}U_j &  V_j
 \end{array}
\right]\in\mathrm{SU}_{n_j},\; U_j^{\dagger}U_j=I_{m_j},\; V_j^{\dagger}V_j=I_{n_j-m_j},\end{equation*}
such that $P_j=\Theta_j\Pi_j\Theta_j^{\dagger}$,  for $j=1,\dots, r.$\\[1mm]
{\bf Initial direction:}
Set $$\widehat{A}_j:=\Psi_{A,j}(P_1,\dots,I_{n_j},\dots,P_r),
$$ compute
$\Psi_{j}',\;\Psi_{j}'',\;\Psi_{j}'''\;\textrm{ as in (\ref{not1})}$ and take the steepest descent direction  $$Z_j= - g_j:=-\Psi_{j}''',$$
for $j=1,\dots, r.$  Denote $Z:=(Z_1,\dots,Z_r),\; g:=(g_1,\dots, g_r).$\\[1mm]
{\bf Step 2.}
{\bf  Stopping criterion:}
$\|\grad_{\rho_A}(P)\|/\rho_A(P)<\varepsilon$.\\[1mm]
{\bf Step 3.} {\bf  QR-updates:}
\begin{equation}
\begin{array}{cc}
\Theta_j^{\mathrm{new}}=\Theta_j\left[ \begin{array}{cc}
\alpha I_{m_j} & \alpha Z_j\\[2mm]
-\alpha Z_j^{\dagger} & \alpha I_{n_j-m_j}
\end{array}\right]_Q, &
P_j=\Theta_j\Pi_{j}\Theta_j^{\mathrm{new}^\dagger},
\end{array}
\end{equation} with the step-size given by
 $\alpha=-a/(b+c)$,
where
 \begin{equation*}
\begin{array}{l}
a := \sum\limits_{j=1}^r\tr\biggl(\Psi_{j}'''Z_{j}^{\dagger}\biggr),\;\;  b
:=\displaystyle\sum\limits_{j=1}^r\tr\biggl(\Psi_{j}'Z_{j}Z_{j}^{\dagger}-Z_{j}\Psi_{j}''Z_{j}^{\dagger}\biggr),\\[4mm]
c:=\sum\limits_{j=1}^{r-1}\sum\limits_{k=j+1}^r  \rho_A(P_1,\dots,\xi_j,\dots,\xi_k,\dots, P_r), \end{array}
\end{equation*}  for $j=1,\dots, r$.  The tangent vectors $\xi_j$ are given in \eqref{xi}.\\[1mm]
 {\bf Step 4.} Set $P:=P^{\mathrm{new}}$ and $\Theta:=\Theta^{\mathrm{new}}$.\\[1mm]
{\bf Step 5.} {\bf  New direction:}
   Update $\Psi_{j}',\;\Psi_{j}'',\;\Psi_{j}'''$ as in (\ref{not1}) and compute the new  direction
\begin{equation}\label{conj_dir}
 Z_j^{\mathrm{new}}= - g_j^{\mathrm{new}}+\beta \; Z_j,\; g_j^{\mathrm{new}}:=\Psi_{j}''',
\end{equation}
for $j=1,\dots, r.$ Here, $\beta$ is given by  the  Polak-Ribiere formula \begin{equation}
\beta=\frac{\langle g^{\mathrm{new}}, g^{\mathrm{new}}-g\rangle}{\langle g, g\rangle}
\end{equation}
{\bf Step 6.} Set $g:=g^{\mathrm{new}},\; Z:=Z^{\mathrm{new}}$ and go to Step 2.\end{minipage}\vspace{0.25cm}}}
\end{table}
 Let $\Theta:=(\Theta_1,\cdots,\Theta_r)$ be such that
$\Theta_kP_k\Theta_k^{\dagger}=\Pi_k.$ 
Furthermore, let $P^{\mathrm{new}}:=(P_1^{\mathrm{new}},\cdots,P_r^{\mathrm{new}})$ denote the updated point in $\mathrm{Gr}^{\times
  r}({\bf m},{\bf n})$ via the QR-coordinates  as in (\ref{up}).
For the computation of the new direction, a  ``transport'' of the old direction  $\xi=(\xi_1,\cdots,\xi_r) $ from $\mathrm{T}_P\mathrm{Gr}^{\times r}({\bf m},{\bf n})$ to the tangent space $\mathrm{T}_{P^{\mathrm{new}}}\mathrm{Gr}^{\times r}({\bf m},{\bf n})$ is required.
We use the following approximation for the paralle transport of $\xi$ along the geodesic through $P$ in direction $\xi$
 \begin{equation}
\begin{array}{cc}
 \xi_j\mapsto\Theta_j^{\mathrm{new}}\left[\begin{array}{cc}
0 & Z_j\\[2mm]
Z_j^{\dagger} & 0
\end{array}\right]\Theta_j^{\mathrm{new}^{\dagger}}, & \textrm{where   } \Theta_j^{\mathrm{new}}=\Theta_j\left[\begin{array}{cc}
I_{m_j} & -Z_j\\[2mm]
Z_j^{\dagger} & I_{n_j-m_j}
\end{array}\right]_Q,
\end{array}
 \end{equation} for all $j=1,\dots,r.$

The complete Riemannian conjugated gradient is presented as Algorithm 2.

It is recommended to reset the search direction to the steepest descent direction
after $d$ iterations, i.e.~$Z_{k}^{\mathrm{new}}:= - g_k^{\mathrm{new}},\; k=1,\dots, r$, where $d$ refers to the dimension of the manifold. For the maximization of the generalized Rayleigh-quotient the initial direction is $Z_k=g_k$ and the update  $Z_k^{\mathrm{new}}=g_k^{\mathrm{new}}+\beta \; Z_k.$

The convergence properties of the NCG methods are
in general difficult to analyze. Yet, under moderate supplementary
assumptions on the cost function one can guarantee that the NCG
converges to a stationary point \cite{nw}.
It is expected that the proposed Riemannian conjugated gradient
method has properties similar to those of the NCG.

\subsection{Numerical experiments}

In this section we run several numerical experiments suitable for
the applications mentioned in Section 3.2, i.e.~best rank approximation for
tensors and subspace clustering, to test the Newton-like
(N-like) and Riemannian conjugated gradient (RCG) algorithms. The
algorithms were implemented in MATLAB on a personal notebook with 1.8 GHz Intel Core 2 Duo processor.

\

\subsubsection{Best multilinear rank-$(m_1,\dots,m_r)$ tensor approximation.}
To test the performance of our algorithms we have considered several examples of large size tensors of order $3$ and $4$ with entries chosen from the standard normal distribution and estimated their best low-rank approximation.
We have started with a truncated HOSVD (\cite{lmv1}) and performed several HOOI iterates before we run our N-like and RCG algorithms.
Depending on the size of the tensor, the number of HOOI iterations necessary to reach the region of attraction of a stationary point $P^*\in\mathrm{Gr}^{\times r}({\bf m},{\bf n})$, ranges from 10 to 100. As stopping criterion we have chosen that the relative norm of the gradient $\|\grad_{\rho_A}(P)\|/\rho_A(P)$ is approximately $10^{-13}$. 

{\bf Computational complexity. }The computational complexity of the N-like method is determined by the computation of the Hessian and the solution of the Newton equation (\ref{sylvester3}). Thus, for the best rank-$(m,m,m)$ approximation of a $n\times n\times n$ tensor, the computation of the Hessian is dominated by tensor-matrix multiplications and is of order $O(n^3m)$. Solving the Newton equation by Gaussian elimination gives a computational complexity of order $O(m^3(n-m)^3)$, i.e. the dimension of the manifold to the power of three.
For the computational costs of the RCG method we have to take into discussion only tensor-matrix multiplications, which give a cost per RCG iteration of order $O(n^3m)$.

{\bf Experimental results and previous work. } The problem of best low-rank tensor approximation has enjoyed a lot of attention recently. Apart from the well known higher order orthogonal iterations -- HOOI (\cite{lmv2}), various algorithms which exploit the manifold structure of the constraint set have been  developed. We refer to \cite{es,ilav2} for Newton methods, to \cite{sl} for quasi-Newton methods and to \cite{ilav1}  for conjugated gradient and trust region methods on the Grassmann manifold.
Similar to the Newton methods in \cite{es,ilav2}, our N-like method converges quadratically to a stationary point of the generalized Rayleigh-quotient when starting in its neighborhood.

We have compared our algorithms with the existing ones in the literature: quasi-Newton with BFGS, Riemannian conjugated gradient method which uses the Armijo-rule for the computation of the step-size (CG-Armijo), and HOOI.
The algorithms were run on the same platform, identically initialized and with the same stopping criterion. For the BFGS quasi-Newton and limited memory quasi-Newton (L-BFGS)  methods we have used the code available in \cite{savas}.

\begin{figure}[htb]
  \begin{center}
\resizebox{63mm}{!}{\includegraphics[width=6.3cm,height=4.5cm]{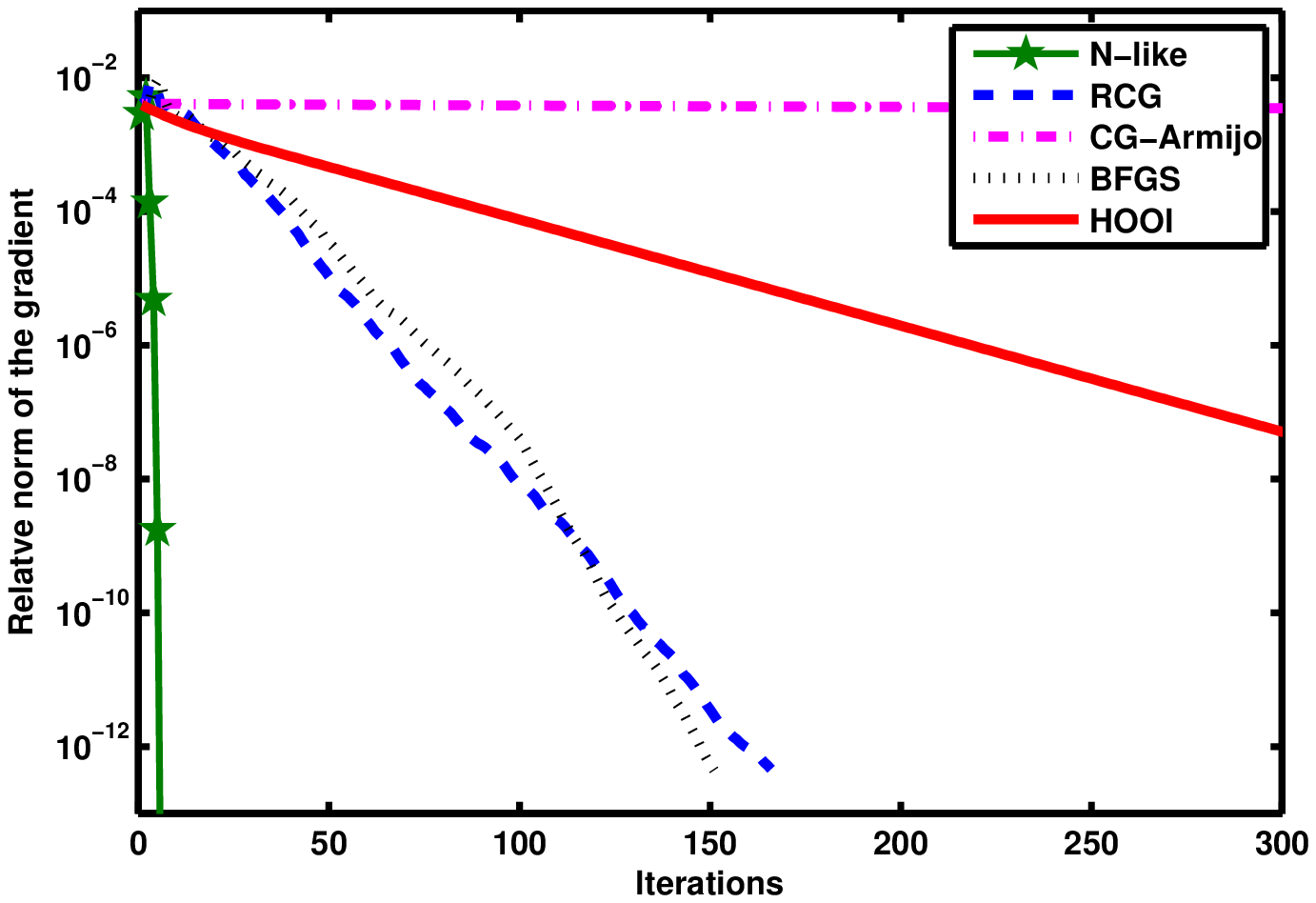}}
\resizebox{63mm}{!}{\includegraphics[width=6.3cm,height=4.5cm]{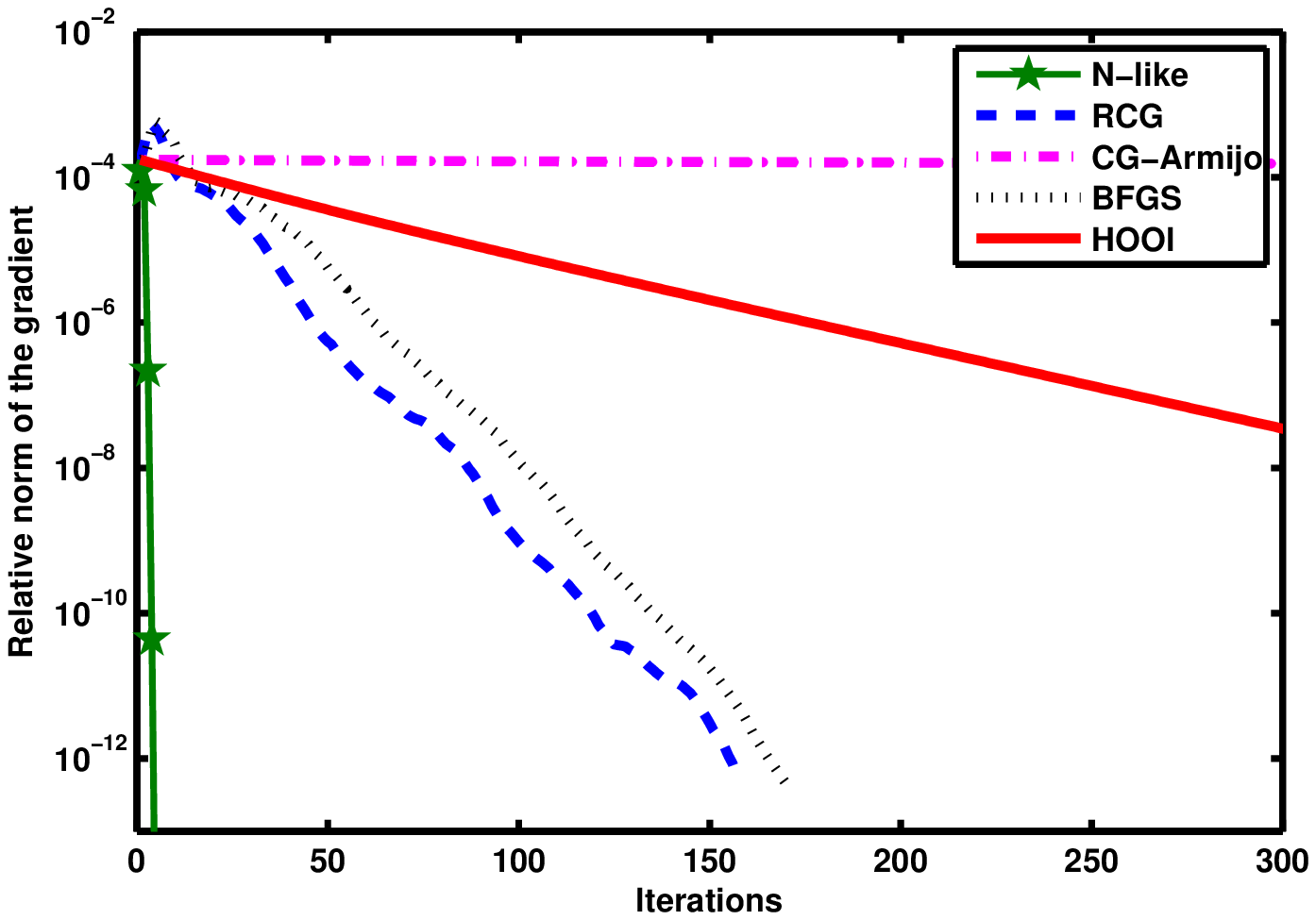}}
 \caption{Convergence for multilinear rank tensor approximation: number of
      iterations versus the relative norm of the gradient $\|\grad_{\rho_A}(P^n)\|/\rho_A(P^n)$ at a
      logarithmic scale.
Left: $100\times 100\times 100$ tensor approximated by a rank-$(5,5,5)$
tensor. Right: $100\times 150\times 200$ tensor approximated by a rank-$(15,10,5)$
tensor. }
    \label{fig1}
  \end{center}
\end{figure}
\begin{figure}[htb]
  \begin{center}
\resizebox{63mm}{!}{\includegraphics[width=6.3cm,height=4.5cm]{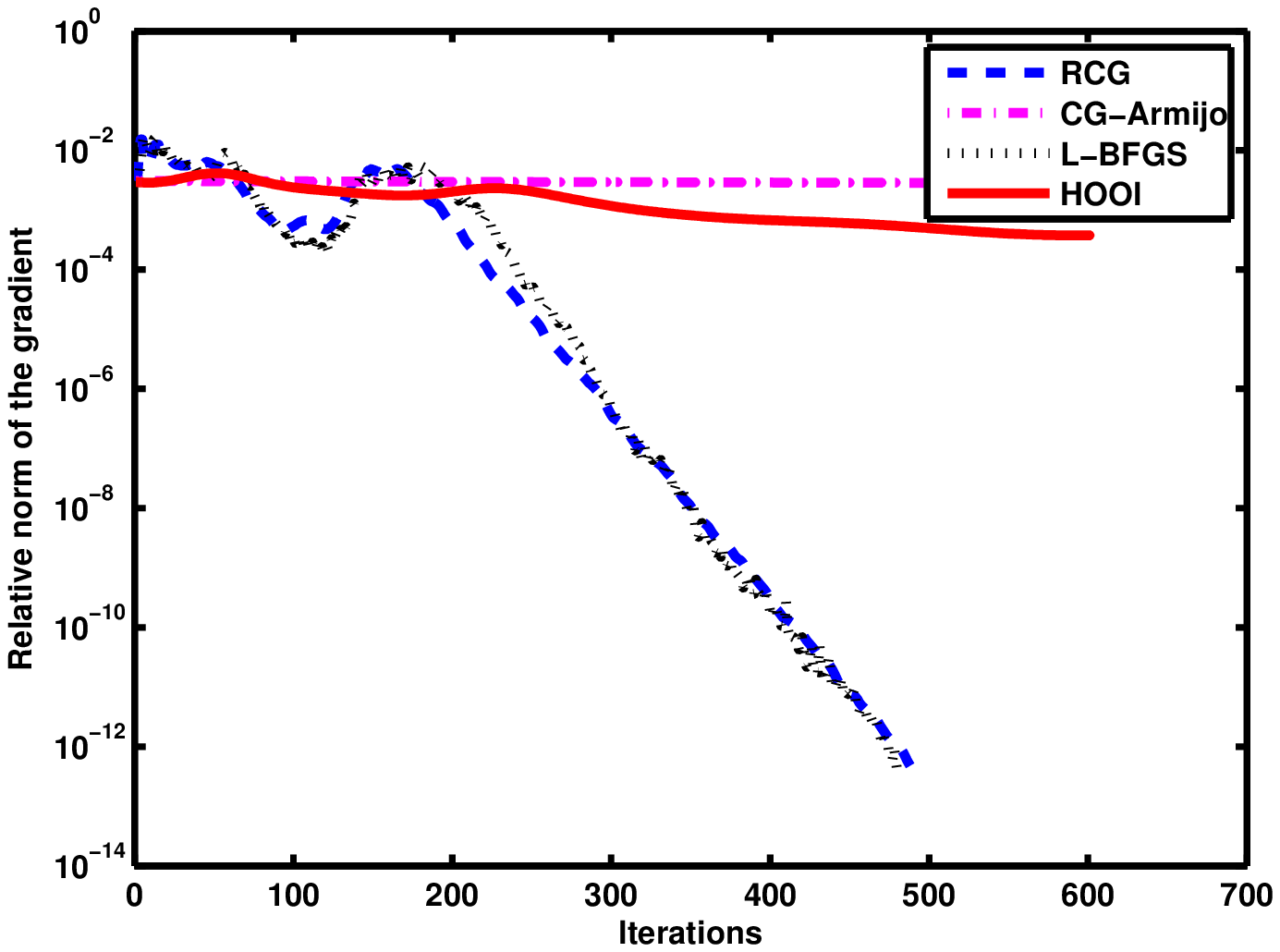}}
   \resizebox{63mm}{!}{\includegraphics[width=6.3cm,height=4.5cm]{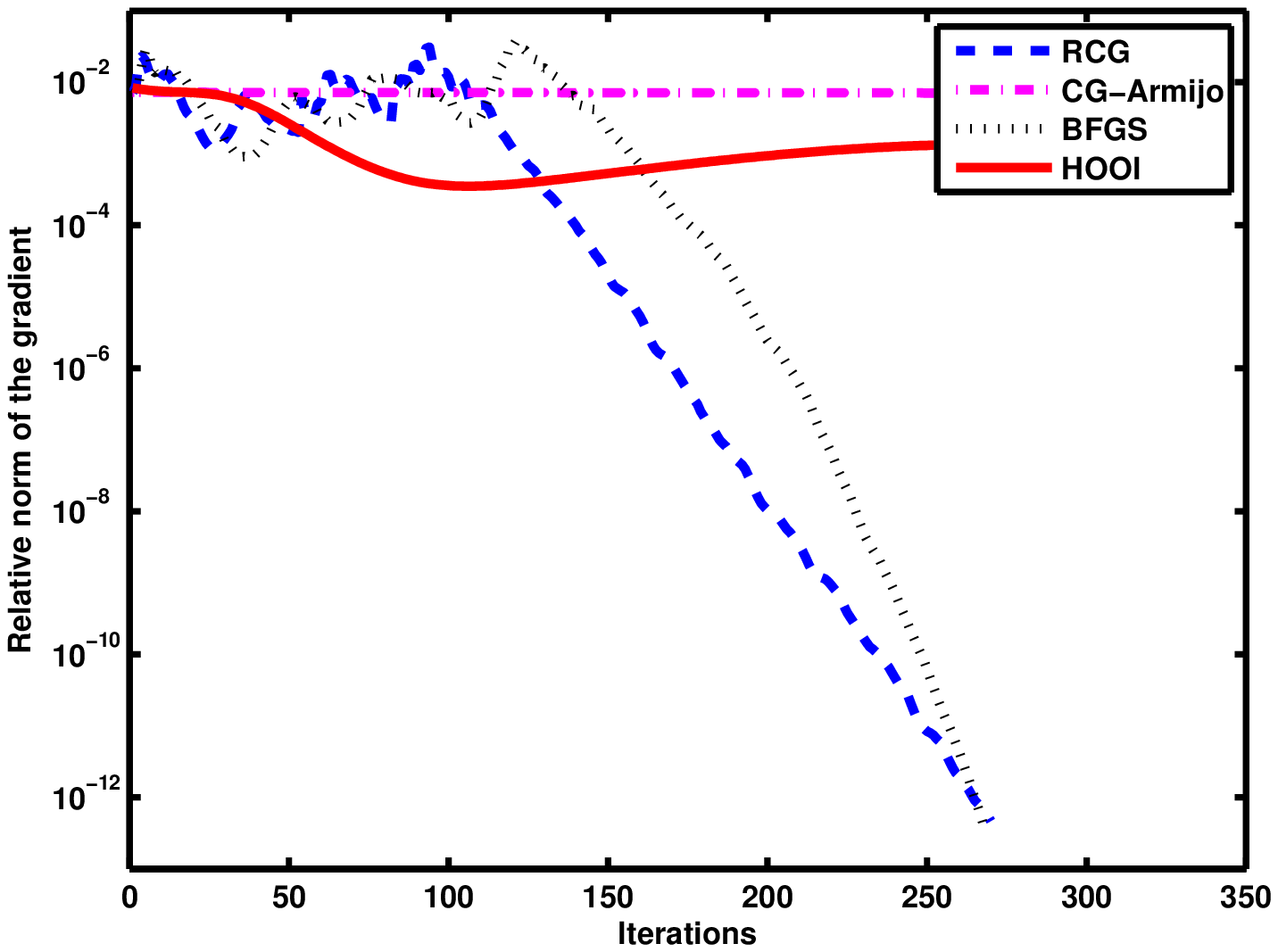}}
 \caption{Convergence for multilinear rank tensor approximation: number of
      iterations versus the relative norm of the gradient $\|\grad_{\rho_A}(P^n)\|/\rho_A(P^n)$ at a
      logarithmic scale.
Left: $200\times 200\times 200$ tensor approximated by a rank-$(10,10,10)$
tensor. Right: $50\times 50\times 50\times 50$ tensor approximated by a rank-$(5,5,5,5)$
tensor. }
    \label{fig11}
  \end{center}
\end{figure}

Fig.~\ref{fig1} shows convergence results for two large size tensors $100\times 100\times 100$
and $100\times 150\times 200$ approximated by rank-$(5,5,5)$ and
rank-$(15,10,5)$ tensors, respectively. In Fig.~\ref{fig11} we plot
the convergence behavior of the RCG method for the best rank-$(10,10,10)$ approximation of a  $200\times 200\times 200$ tensor (left) and for the best rank-$(5,5,5,5)$ approximation of a $50\times 50\times 50\times 50$ tensor.  Due to the limited memory space, we were not able to run the N-like and BFGS quasi-Newton algorithms for the example on the left. Yet it was still possible to run RCG, L-BFGS, CG-Armijo and HOOI.

In Table~\ref{tab:1} we display the
average CPU times necessary to compute a low rank best
approximation for tensors of different sizes and orders
by N-like, RCG, BFGS and L-BFGS quasi-Newton methods. We have performed 100 runs for each example.

\begin{table}[h]
\centering
\caption{\small{Average CPU Time}}
    \begin{tabular}{ | l | c | c |c|c|}
\toprule
\textbf{Tensor size and rank} & {\bf N-like} & {\bf  RCG} & {\bf BFGS} & {\bf L-BFGS}\\[1mm]
 \hline
\toprule
$50\times 50\times 50,\;$ rank-$(7, 8, 5)$ & 2 s  & 6 s  & 24 s  & 13 s  \\ \hline
$100\times 100\times 100,\;$ rank-$(5, 5, 5)$ & 70 s & 75 s & 150 s  & 94 s\\
\hline
$200\times 200\times 200,\;$ rank-$(5, 5, 5)$ & -  & 11 min & - & 14 min \\ \hline
$50\times 50\times 50\times 50,\;$ rank-$(5, 5, 5, 5)$ & -  & 9 min & 11 min  & -  \\
 \hline
\bottomrule
\end{tabular}
\label{tab:1}
\end{table}

{\bf Resume. } First we mention that there is no guarantee that the N-like and RCG iterations converge to a local maximizer of the generalized Rayleigh-quotient.  However, in the examples presented  in Fig.\ref{fig1} and Fig.\ref{fig11} the limiting points are local maximizers. As the numerical experiments have shown, the N-like method has the advantage of fast convergence. Unfortunately, for large scale problems, the N-like algorithm can not be applied, as mentioned before. Even when it is possible to apply N-like algorithm, it needs a large amount of time per iteration. As an example, for  the best rank-$(10,10,10)$ of a $180\times 180\times 180$ tensor, one N-like iteration took  three minutes.
Related algorithms which explicitly compute the Hessian and solve the Newton equation, such as  \cite{es, ilav2}, and those which approximately solve the Newton equation such as the trust region method   \cite{ilav1}, face the same difficulty for large scale problems. 
 On the other hand, the low cost iterations of the RCG method makes it a good candidate to solve large size problems. The convergence rate is comparative to that of the BFGS quasi-Newton method in \cite{savas}, but at much lower computational costs.  Our experiments exhibit the shortest CPU time for the RCG method. In the examples in which the tensor was a small perturbation of a low-rank tensor, the RCG algorithm exhibits quadratic convergence.

\

\subsubsection{Subspace Clustering}  The experimental setup consists in choosing $r$
subspaces in $\R^3$ ($r=2,3$ and $4$) and collections of $200$
randomly chosen\footnote{The points have been generated by fixing an
orthogonal basis within the subspaces and choosing corresponding coordinates
randomly with a uniform distribution over the interval $[-5,5]$.}
points on each subspace. Then, the sample points are perturbed by
adding zero-mean Gaussian noise with standard deviation
varying from $0\%$ to $5\%$ in the different experiments.
Now, the goal is to detect the exact subspaces or to approximate
them as good as possible. For this purpose, we apply our N-like and
RCG algorithms to solve the associated optimization task, cf. Section 3.1.
The error between the exact subspaces
and the estimated ones is measured as in \cite{vmp}, i.e.
\begin{equation}
\mathrm{err}:=\frac{1}{r}\sum\limits_{j=1}^{r}\mathrm{arccos}\biggl(\frac{1}{m_j^2}| \tr(P_j\tilde{P}_j)|\biggr),
\end{equation}
where $P_j$ is the orthogonal projector corresponding to the exact
subspace and $\tilde{P}_j$ the orthogonal projector corresponding
to the estimated one.
\begin{figure}[htb]
  \begin{center}
    \resizebox{63mm}{!}{\includegraphics[width=6.3cm,height=5cm]{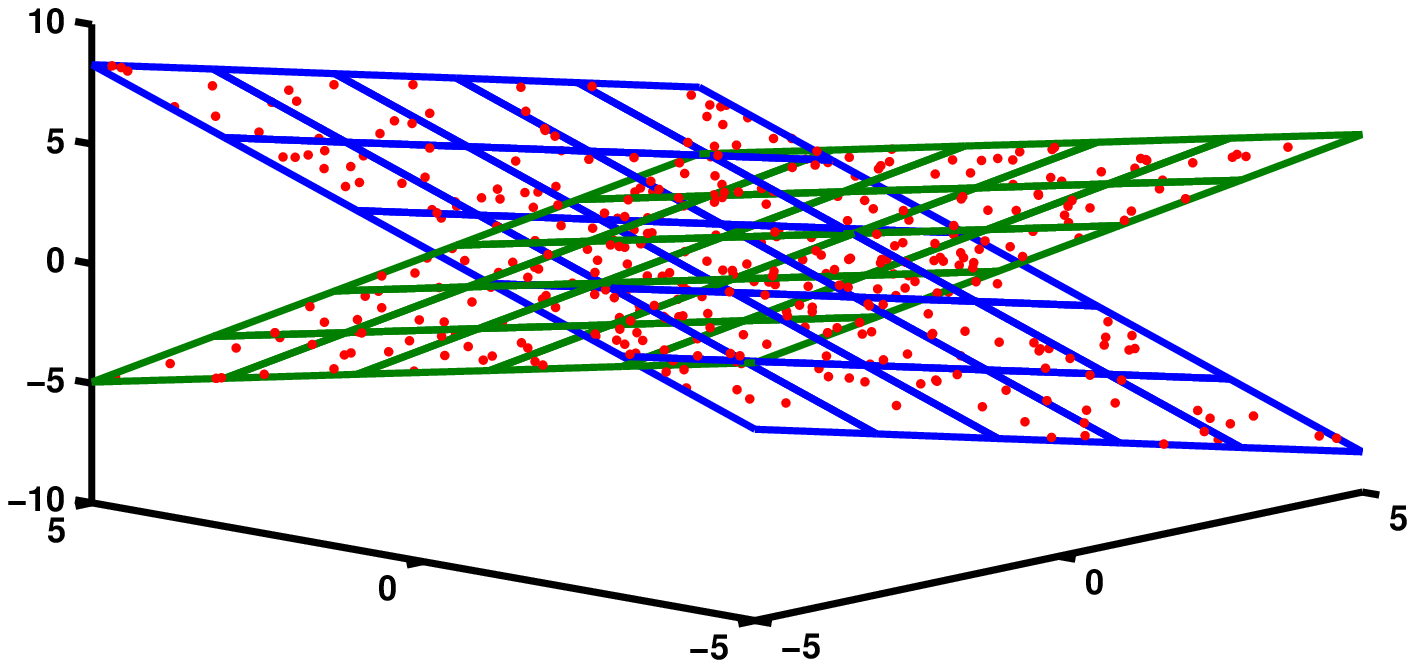}}
     \resizebox{63mm}{!}{\includegraphics[width=6.3cm,height=5cm]{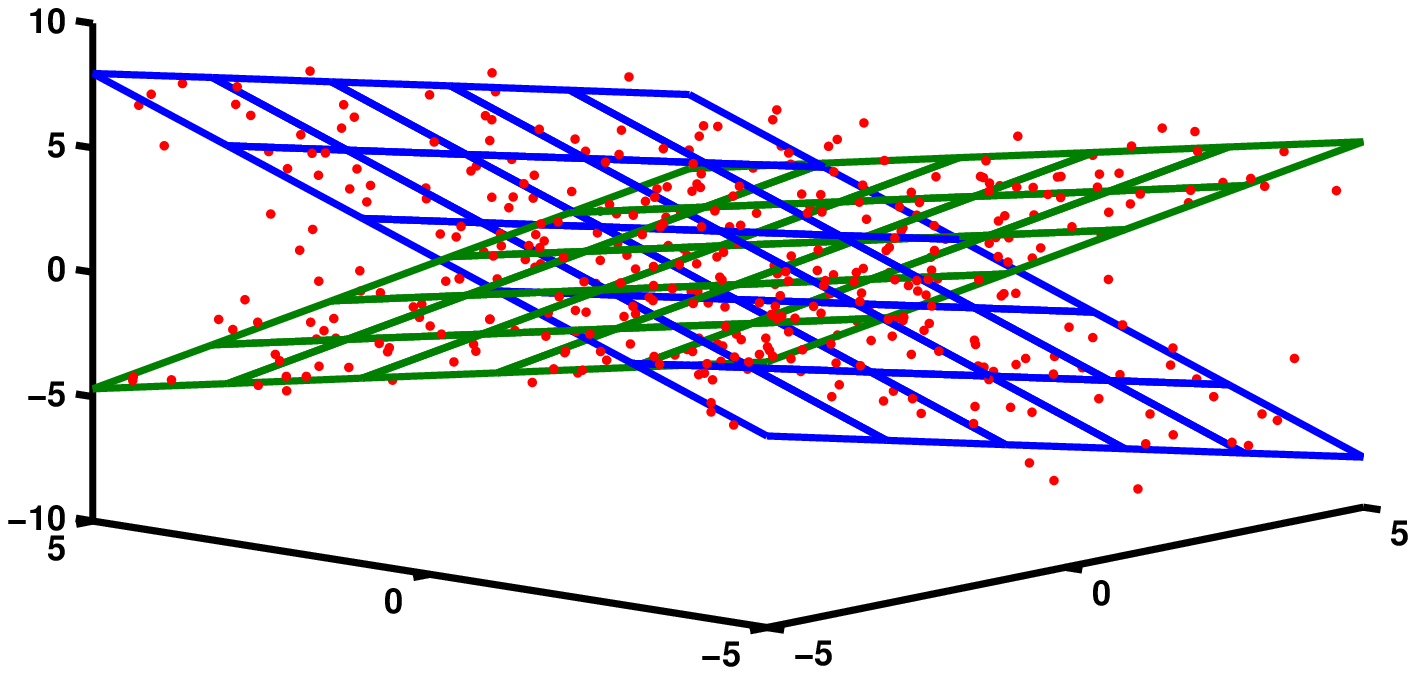}}
    \caption{Left: Data points drawn from the union of  two subspaces of
      dimension $2$ (through the origin) of $\R^3$.  Right: Data  points from the left figure slightly perturbed by zero mean Gaussian
     noise with $5\%$ standard deviation.}
    \label{fig2}
  \end{center}
\end{figure}

It can be easily checked that in the case of unperturbed data there is a unique non-degenerate minimizer of $\rho_A$, and it yields the exact subspaces. Thus,
we expect that for noisy
data the global minimizer still gives a good approximation. Since $\rho_A$ has
many local optima, for an arbitrary starting point our algorithms can
converge to stationary points which lead to a significant error
between the exact subspaces and their approximation.
Thus, in what follows, we briefly describe a method (PDA, see below)
for computing a suitable initial point which guarantees the
convergence of our algorithms towards a good approximation of
the exact subspaces in our numerical experiment:

\smallskip
The Polynomial Differential Algorithm (PDA) was proposed in
\cite{vmp}. It is a purely algebraic method for
recovering a finite number of subspaces from a
set of data points belonging to the union of these subspaces.
 From the data set finitely many homogeneous polynomials
are computed such that their zero set coincides with the union of the sought subspaces.
Then, an evaluation of their derivatives
 at given data points yields successively a basis of the orthogonal
complement of subspaces one is interested in.
For noisy data, a slightly modified version of PDA \cite{vmp}
yields an approximation of the unperturbed subspaces. This ``first"
approximation turned out to be a good starting point for our
iterative algorithms which significantly improved the
approximation quality.

\smallskip
For each noise level we perform 500 runs of the N-like and Local-CG algorithms
for different data sets and compute the mean error between the exact subspaces
and the computed approximations. As a preliminary step, we normalize all data points, such that no direction is favored.

In Fig.~\ref{fig2}, $400$ randomly chosen data points which
lie exactly in the union of two $2$-dimensional subspaces of $\R^3$
(left)  and their perturbed\footnote{Gaussian noise with $5\%$ standard
deviation} images (right)  are depicted. Moreover, the two plots display
the exact subspaces (left) as well as the ones computed
by our N-like algorithm (right). The error between the exact
subspaces and our approximation is ca.~$2^{\circ}$, whereas
the error for the PDA approximation is ca.~$5^{\circ}$.
\begin{figure}[htb]
  \begin{center}
    \resizebox{63mm}{!}{\includegraphics[width=6.3cm,height=5cm]{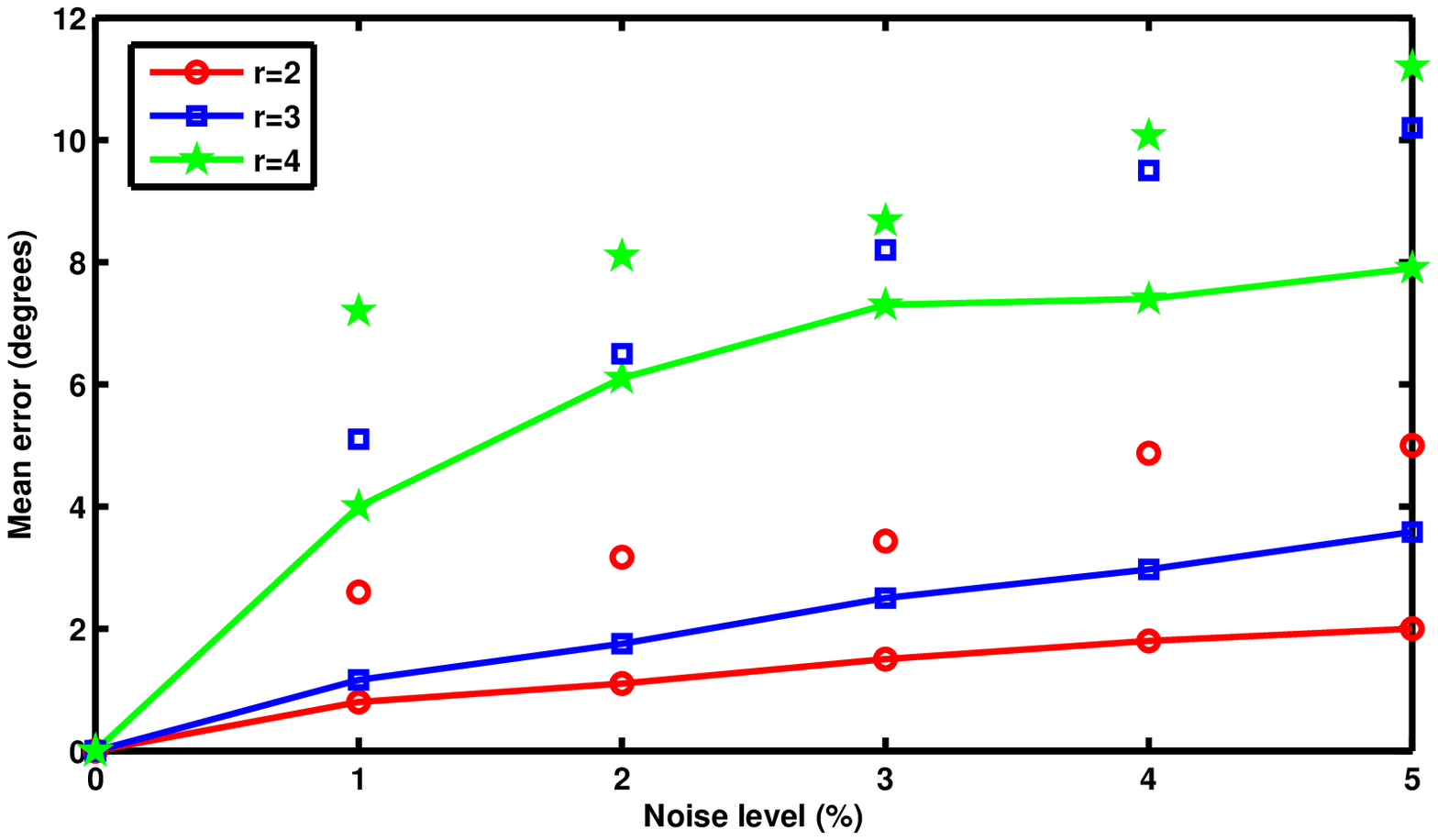}}
    \resizebox{63mm}{!}{\includegraphics[width=6.3cm,height=5cm]{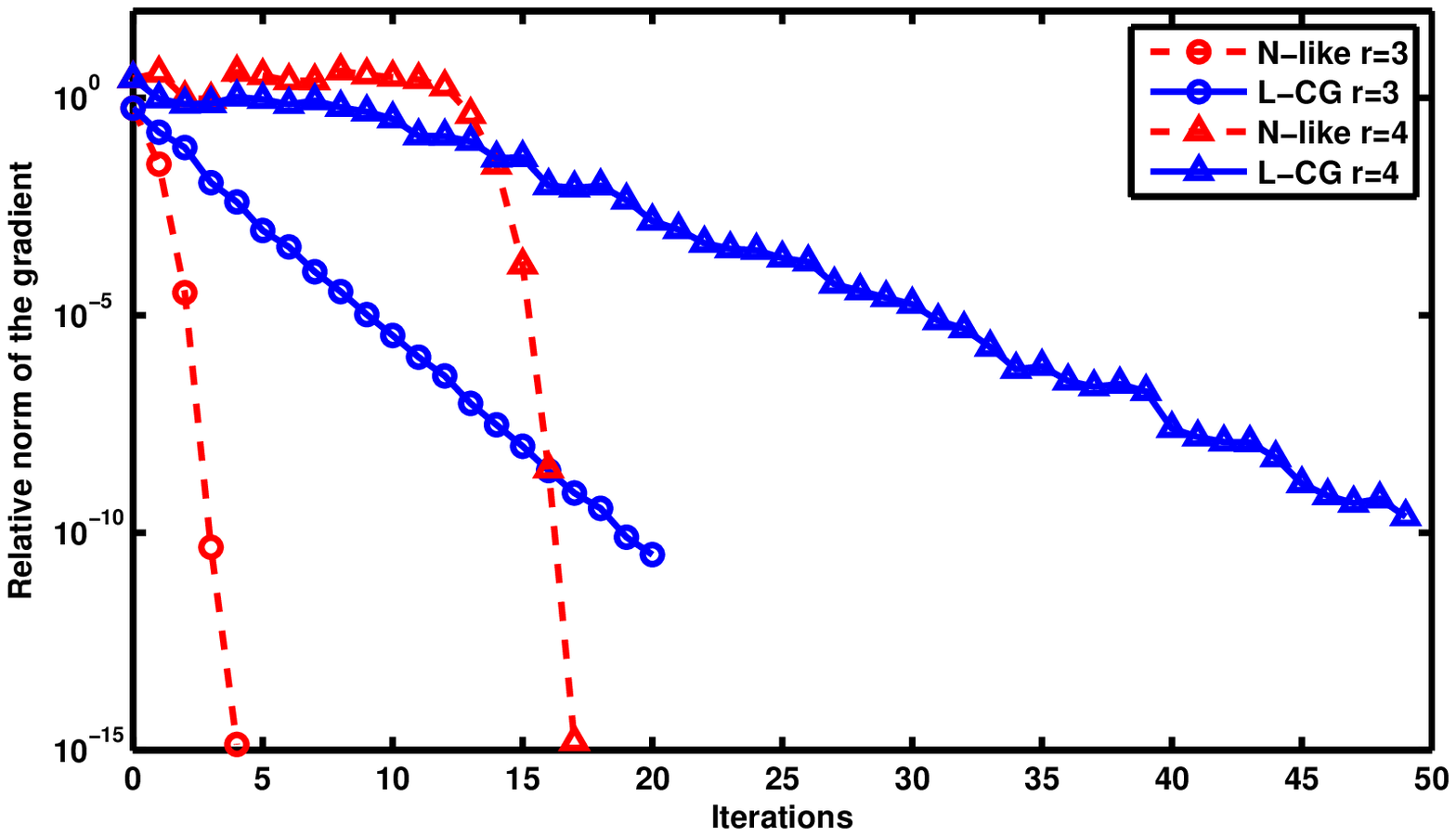}}
      \caption{Left: The mean error for noise
    levels from $0\%$ to $5\%$ and different number of subspaces. The
    disconnected symbols refer to the initial error (PDA) and the
    corresponding continuous lines refer to the error estimated by our algorithms. Right: Convergence of N-like and RCG for subspace clustering: number of
      iterations versus the relative norm of the gradient $\|\grad_{\rho_A}(P^n)\|/\rho_A(P^n)$ at a
      logarithmic scale. Data points from $3$ and resp. $4$ subspaces perturbed
    with $5\%$ Gaussian noise. Average CPU time:
ca.~$0.4$ and ca.~$2$ seconds for the N-like and RCG algorithm,
respectively (1.8 GHz Intel Core 2 Duo processor).}
    \label{fig3}
  \end{center}
\end{figure}

In Fig.~\ref{fig3}, we plot the mean error (left) for
different noise levels and different number of subspaces.
We have included also the mean error for the starting point
of our algorithms, i.e.~for the PDA approximation.
On the right we demonstrate the fast convergence
rate of the N-like and RCG algorithms for the case of $3$ and,
respectively, $4$ subspaces.

\textbf{Resume.} Our numerical experiments have proven that (i)
the minimization task proposed in Section 3 is capable to
solve subspace detection problems and (ii) our numerical algorithms
initialized with the PDA starting point yield an effective
method for computing a reliable approximation of the perturbed
subspaces. How the approximation of the perturbed subspaces varies when the noise in the data follows some law of distribution, is the subject of future investigation.

\section{Appendix}

\

Here we provide a proof of  Proposition $\ref{diffeo}$, which states that there exists a global Riemannian isometry $\varphi$ between $\mathrm{Gr}^{\otimes r}({\bf m},{\bf n})$ and $\mathrm{Gr}^{\times r}({\bf m},{\bf n}).$

\begin{proof}
The surjectivity of $\varphi$ is clear from the definition of
$\mathrm{Gr}^{\otimes r}({\bf m},{\bf n})$.
To prove the injectivity of $\varphi$ we use induction over $r$.
Choose  $(P_1,..,P_r),\; (Q_1,\dots,Q_r)\in
\mathrm{Gr}^{\times r}({\bf m},{\bf n})$ such that $P_1\otimes
\dots\otimes P_r=Q_1\otimes \dots\otimes Q_r,$ i.e.
\begin{equation}\alpha_{ij}P_r=\beta_{ij}Q_r
\quad\textrm{for all}\; i,\;j,\end{equation} where $\alpha_{ij}$ and
$\beta_{ij}$ are the entries of $P_1\otimes
\dots\otimes P_{r-1}$ and $Q_1\otimes \dots\otimes Q_{r-1}$, respectively.\\
Thus it exists $\gamma\in\C$ such that $P_r=\gamma Q_r$. Since $P_r$ and $Q_r$
have only $0$ and $1$ as eigenvalues it follows that $\gamma=1$ and $P_r=Q_r$.
Therefore, $P_1\otimes \dots\otimes P_r=Q_1\otimes \dots\otimes Q_r$ implies that
\begin{equation}P_1\otimes\dots\otimes P_{r-1}=Q_1\otimes \dots\otimes Q_{r-1}\end{equation} and the
procedure can be repeated until we obtain $P_j=Q_j$, for all $j=1,\dots,r$. Thus the injectivity of $\varphi$ is proven. So $\varphi$ is a continuous
bijective map with continuous  inverse $\varphi^{-1}$ due to the compactness
of $\mathrm{Gr}^{\times r}({\bf m},{\bf n})$. Moreover, the map $\varphi$ is smooth since the components of $P_1\otimes
\dots\otimes P_{r}$ are polynomial functions.
Let $P:=(P_1,\dots, P_r)$ and ${\bf P}:=P_1\otimes \dots\otimes P_r$.
Consider the tangent map of $\varphi$ at $P$, i.e.
\begin{equation}
\begin{array}{c}
D\varphi(P) : \mathrm{T}_P \mathrm{Gr}^{\times r}({\bf m},{\bf n})\rightarrow \mathrm{T}_{\bf P} \mathrm{Gr}^{\otimes r}({\bf m},{\bf n}),\\[2mm]
(\xi_1,\dots, \xi_r)\mapsto \displaystyle\sum\limits_{j=1}^r P_1\otimes\dots\otimes \xi_j\otimes \dots\otimes P_r.
\end{array}
\end{equation}
With the inner products (\ref{inner1})  and (\ref{innerprod3}) defined  on $\mathfrak{her}_N$ and $\mathfrak{her}_{n_1}\times\cdots\times\mathfrak{her}_{n_r}$, respectively, one has
 \begin{equation}\label{varphi}
\left<D\varphi(P)\xi, D\varphi(P)\eta\right>
= \displaystyle\sum\limits_{j=1}^r
M_j\tr(\xi_j\eta_j)= \langle \xi, \eta\rangle,
\quad
M_j:=\prod\limits_{k=1,\, k\neq j}^rm_k.
\end{equation}  This
implies that the tangent map $D\varphi(P)$ is a linear isometry. Thus,
it is
invertible and therefore $\varphi$ is a local
diffeomorphism. Moreover, since $\varphi$ is bijective it is a global
diffeomorphism, giving thus a global Riemannian isometry
when the metric on $\mathrm{Gr}^{\times r}({\bf m},{\bf n})$
is defined by \eqref{innerprod4}.
\end{proof}

\

\section*{ Acknowledgements}
This work has been supported by the Federal Ministry of Education and
Research (BMBF) through the project  FHprofUnd 2007: Cooperation Program
between Universities of Applied Science and Industry,
"Development and Implementation of Novel Mathematical Algorithms
for Identification and Control of Technical Systems".

\end{document}